\documentclass[onefignum,onetabnum]{siamonline220329}

\usepackage{lipsum} 
\usepackage{amsfonts}
\usepackage{graphicx}
\usepackage{epstopdf}
\usepackage{algorithmic}
\ifpdf
  \DeclareGraphicsExtensions{.eps,.pdf,.png,.jpg}
\else
  \DeclareGraphicsExtensions{.eps}
\fi

\usepackage{enumitem}
\setlist[enumerate]{leftmargin=.5in}
\setlist[itemize]{leftmargin=.5in}

\newsiamremark{remark}{Remark}
\newsiamremark{hypothesis}{Hypothesis}
\crefname{hypothesis}{Hypothesis}{Hypotheses}
\newsiamthm{claim}{Claim}

\headers{An \textit{a posteriori} error estimate for a 0D/2D coupled model}{H. Albazzal, A. Lozinski, and R. Tittarelli}

\title{An \textit{a posteriori} error estimate for a 0D/2D coupled model
\thanks{\funding{This work was supported by the French ‘Investissements d’Avenir’ program, project Agence
Nationale de la Recherche (ISITE-BFC) (contract ANR-15-IDEX-0003). }}}

\author{
Hussein Albazzal\thanks{Université de Franche-Comté, CNRS, LmB, F-25000 Besançon, France 
  (\email{hussein.albazzal@univ-fcomte.fr}).}
\and Alexei Lozinski\thanks{Université de Franche-Comté, CNRS, LmB, F-25000 Besançon, France
  (\email{alexei.lozinski@univ-fcomte.fr}).}
\and Roberta Tittarelli\thanks{
SUPMICROTECH, CNRS, institut FEMTO-ST, F-25000 Besançon, France (\email{roberta.tittarelli@ens2m.fr}).}
}

\usepackage{amsopn}

\usepackage{subcaption}
\usepackage{nicefrac}
\usepackage{pgfplots,pgfplotstable}
\usepackage{tikz}
\usepackage{tkz-euclide}
\newcommand{\logLogSlopeTriangle}[5]
{
    \pgfplotsextra
    {
        \pgfkeysgetvalue{/pgfplots/xmin}{\xmin}
        \pgfkeysgetvalue{/pgfplots/xmax}{\xmax}
        \pgfkeysgetvalue{/pgfplots/ymin}{\ymin}
        \pgfkeysgetvalue{/pgfplots/ymax}{\ymax}

        \pgfmathsetmacro{\xArel}{#1}
        \pgfmathsetmacro{\yArel}{#3}
        \pgfmathsetmacro{\xBrel}{#1-#2}
        \pgfmathsetmacro{\yBrel}{\yArel}
        \pgfmathsetmacro{\xCrel}{\xArel}

        \pgfmathsetmacro{\lnxB}{\xmin*(1-(#1-#2))+\xmax*(#1-#2)} 
        \pgfmathsetmacro{\lnxA}{\xmin*(1-#1)+\xmax*#1} 
        \pgfmathsetmacro{\lnyA}{\ymin*(1-#3)+\ymax*#3} 
        \pgfmathsetmacro{\lnyC}{\lnyA+#4*(\lnxA-\lnxB)}
        \pgfmathsetmacro{\yCrel}{\lnyC-\ymin)/(\ymax-\ymin)}

        \coordinate (A) at (rel axis cs:\xArel,\yArel);
        \coordinate (B) at (rel axis cs:\xBrel,\yBrel);
        \coordinate (C) at (rel axis cs:\xCrel,\yCrel);

        \draw[#5]   (A)-- node[pos=0.5,anchor=north] {\scriptsize{1}}
                    (B)-- 
                    (C)-- node[pos=0.,anchor=west] {\scriptsize{#4}} 
                    cycle;
    }
}
\newcommand{\logLogSlopeTriangleBIS}[5]
{
    \pgfplotsextra
    {
        \pgfkeysgetvalue{/pgfplots/xmin}{\xmin}
        \pgfkeysgetvalue{/pgfplots/xmax}{\xmax}
        \pgfkeysgetvalue{/pgfplots/ymin}{\ymin}
        \pgfkeysgetvalue{/pgfplots/ymax}{\ymax}

        \pgfmathsetmacro{\xBrel}{#1}
        \pgfmathsetmacro{\yBrel}{#3}
        \pgfmathsetmacro{\xArel}{#1-#2}
        \pgfmathsetmacro{\yArel}{\yBrel}
        \pgfmathsetmacro{\xCrel}{\xArel}

        \pgfmathsetmacro{\lnxB}{\xmin*(1-#1)+\xmax*#1} 
        \pgfmathsetmacro{\lnxA}{\xmin*(1-(#1-#2))+\xmax*(#1-#2)} 
        \pgfmathsetmacro{\lnyB}{\ymin*(1-#3)+\ymax*#3} 
        \pgfmathsetmacro{\lnyC}{\lnyB+#4*(\lnxA-\lnxB)}
        \pgfmathsetmacro{\yCrel}{\lnyC-\ymin)/(\ymax-\ymin)}

        \coordinate (A) at (rel axis cs:\xArel,\yArel);
        \coordinate (B) at (rel axis cs:\xBrel,\yBrel);
        \coordinate (C) at (rel axis cs:\xCrel,\yCrel);
        \pgfmathsetmacro{\absoluteSlope}{-1*(#4)}

        \draw[#5]   (B)-- 
                    (C)-- node[pos=0.5,anchor=east] {\scriptsize{\absoluteSlope}} 
                    (A)-- node[pos=0.5,anchor=north] {\scriptsize{1}}
                    cycle;
    }
}

\newcommand{\Omegap}{\Omega^{\prime}}
\newcommand{\Omegat}{\widetilde{\Omega}}
\newcommand{\GammaIn}{\Gamma_{\mathrm{in}}}
\newcommand{\GammaOut}{\Gamma_{\mathrm{out}}}
\newcommand{\GammaWall}{\Gamma_{\mathrm{wall}}}
\newcommand{\GammaWallt}{\widetilde{\Gamma}_{\mathrm{wall}}}

\newcommand{\R}{\mathbb{R}} 
\newcommand{\Po}[1]{\mathbb{P}_{#1}}
\newcommand{\LAPL}{\mathop{}\!\mathbin\bigtriangleup} 
\newcommand{\DIV}{\mathrm{div}} 

\newcommand{\Th}{\mathcal{T}_{h}}
\newcommand{\Nh}{\mathcal{V}_{h}}
\newcommand{\NhGamma}{\mathcal{V}_{h}^{\Gamma}}
\newcommand{\NhWGamma}{\Nh\backslash\NhGamma}
\newcommand{\n}{\mathrm{a}}
\newcommand{\SigmaN}{\mathbf{\Sigma}_{h}^{\n}}
\newcommand{\SigmaGammah}{\mathbf{\Sigma}_{h}^{\Gamma}}

\newcommand{\QN}{\mathcal{Q}_{h}^{\n}}
\newcommand{\QGamma}{\mathcal{Q}_{h}^{\Gamma}}
\newcommand{\tauh}{\mathbf{\tau}_{h}}
\newcommand{\K}{T}

\newcommand{\uAv}{u_{\mathrm{av}}}  
  
\newcommand{\uOut}{u_\mathrm{out}}

\newcommand{\tu}{\tilde{u}} 
\newcommand{\up}{u^{\prime}} 
\newcommand{\uht}{\tilde{u}_h} 
\newcommand{\uhc}{u_{h}^{c}}
\newcommand{\uc}{u^{c}} 
\newcommand{\sigmahc}{\mathbf{\sigma}_{h}^{c}} 
\newcommand{\sigmaGammah}{\mathbf{\sigma}_{h}^{\Gamma}} 
\newcommand{\sigmaGammaCont}{\mathbf{\sigma}^{\Gamma}} 
\newcommand{\sigmaGammaSD}{\mathbf{\hat{\sigma}}^{\Gamma}} 
\newcommand{\sigmaN}{\mathbf{\sigma}_{h}^{\n}} 
\newcommand{\sigmaP}{\mathbf{\sigma^{'}}}
\newcommand{\sigmat}{\mathbf{\tilde{\sigma}}_{h}}
\newcommand{\err}{e}  
\newcommand{\bn}{\mathbf{n}}

\newcommand{\norm}[2][]{\left\|#2\right\|_{#1}} 
\newcommand{\tmop}[1]{\ensuremath{\operatorname{#1}}}

\usepackage{mathtools} 
\usepackage{algorithmic,amssymb} 
\Crefname{ALC@unique}{Line}{Lines} 

\ifpdf
\hypersetup{
  pdftitle={An \textit{a posteriori} error estimate for a 0D-2D coupled model},
  pdfauthor={H. Albazzal, A. Lozinski, and R. Tittarelli}
}
\fi
\begin{document}
\maketitle

\begin{abstract}
This work is motivated by the need of efficient numerical simulations of gas flows in the serpentine channels used in proton-exchange membrane fuel cells. In particular, we consider the Poisson problem in a 2D domain composed of several long straight rectangular sections and of several bends corners. In order to speed up the resolution, we propose a 0D model in the rectangular parts of the channel  and a Finite Element resolution in the bends. To find a good compromise between precision and time consuming, the challenge is double: how to choose a suitable position of the interface between the 0D and the 2D models and how to control the discretization error in the bends. We shall present an \textit{a posteriori} error estimator based on an equilibrated flux reconstruction in the subdomains where the Finite Element method is applied. The estimates give a global
upper bound on the error measured in the energy norm of the difference between the
exact and approximate solutions on the whole domain. They are guaranteed, meaning that they feature no undetermined constants. (global) Lower bounds for the error are also
derived. An adaptive algorithm is proposed to use smartly the estimator for aforementioned double challenge. A numerical validation of the estimator and the algorithm completes the work. \end{abstract}



\section{Introduction}

The present work is motivated by models of the serpentine cathode-anode flow channels in Proton Exchange Membrane Fuel Cells (PEMFC), cf. \cite{karvonen2006modeling}. 
One of the computationally intensive tasks in this modeling is to solve the steady-state incompressible Navier-Stokes equations describing the gas flow in the long channels, characterised by very stretched rectangular regions linked by relatively small bends, cf. \cref{fig:RealThing}. 
Typically, one prescribes the Poiseuille flow as boundary conditions on the inlet and outlet boundaries of the channel, and no-slip boundary condition on the wall, cf. for example \cite{wen2010numerical}. 
In order to speed up the computations, the idea is to develop a coupled model as follows (here, in the 2D setting): 
in the rectangular regions of the domain, the flow is approximated by simple analytical solutions, namely the Poiseuille flow which is very accurate sufficiently far from the bends (we call this  0D model), while keeping the original governing equations (the 2D model) in the bend regions. We refer to this as the 0D/2D model. 
There exists different ways to derive a coupled model. 
In \cite{quarteroni2004mathematical,formaggia1999multiscale,formaggia2001coupling},  0D/3D coupling is obtained for the time dependent Navier-Stokes system by integrating the governing equations on a section and substituting an appropriate closure approximation. 
In \cite{gerbeau:inria-00072549} and in \cite{miglio2005model} an asymptotic analysis is used to get respectively the 1D/3D and 1D/2D coupled simplified models for time dependent Navier-Stokes equations. 

\begin{figure}
	\centering
	\includegraphics[scale=0.5]{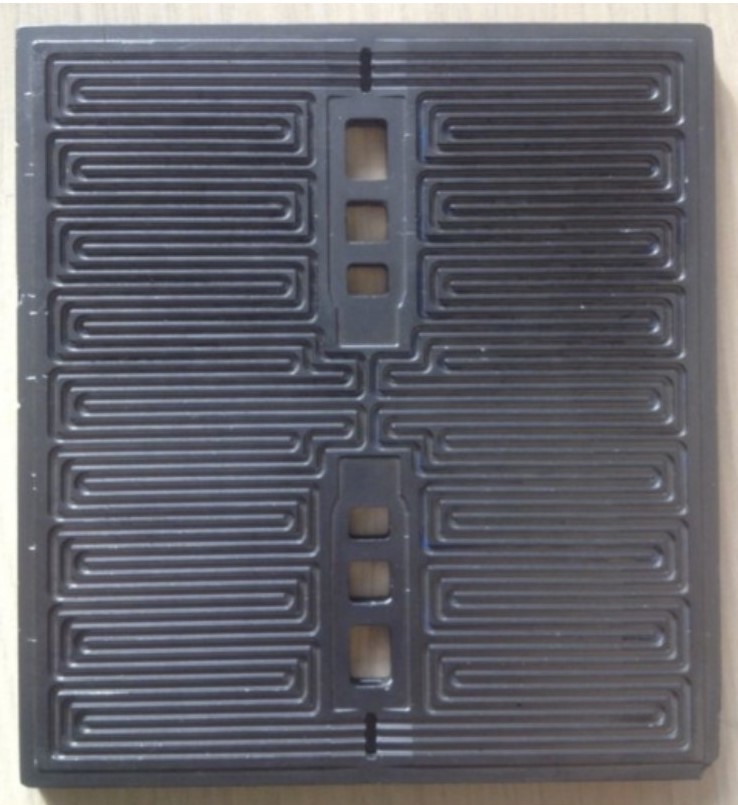}
	\caption{Gaz channels in a flow plate of a fuel cell \cite{zhou2017tridiagonal}.}\label{fig:RealThing}
\end{figure}

In this paper, we consider a simple toy model, replacing the steady Navier-Stokes system by the Poisson equation, to test the basic ideas of the 0D/2D model coupling. 
A 0D/2D coupled model for Poisson equation is derived in \cite{tayachi2014design} and \cite{panasenko1998method} by the asymptotic analysis. In the stretched rectangular portion of the domain $\Omega$, which we call in this paper $\Omegap$ and refer to as the 0D domain, cf. Fig.~\ref{fig:coupledDomain}, the solution is approximated by a simple explicit expression (the analogue of the Poiseuille velocity). In the remaining part of the domain, denoted by $\Omegat=\Omega\setminus\overline{\Omegap}$ and referred to as the 2D domain, we approximate the solution of the original Poisson equation (the 2D model) using the Finite Element (FE) method. 
The 0D and 2D domains are  separated by an interface $\Gamma$. 
We use a simple coupling condition  on the interface $\Gamma$: we impose the continuity of the coupled solution on $\Gamma$. 
The error between the ``original 2D model'' and the ``0D/2D coupled discretized model'' comes thus from two sources: the FE discretization in $\Omegat$ and the coupling error provoked by the simple approximation in $\Omegap$ and controlled by the position of the interface $\Gamma$. 

In the present paper, we propose an \textit{a posteriori} error estimator for the 0D/2D coupled model which can be interpreted as a sum of two contributions: the first one measuring the error due to the simplification introduced by the 0D domain $\Omegap$ and the second one measuring  the discretization error in the 2D domain $\Omegat$. Equilibrating the two contributions, under a given tolerance, enables us to determine the position of the interface $\Gamma$, reducing the size of the 2D domain $\Omegat$ to be discretized, and to construct an optimized mesh on $\Omegat$.
Our error estimator is based on the flux reconstruction technique as in \cite{ern2015polynomial, ern:hal-01377007}. The originality of our work consists in a new flux reconstruction $\sigma_h$, which is defined on the whole domain $\Omega=\Omegap\cup\Gamma\cup\Omegat$, i.e. both on the discretized (2D) and the non-discretized (0D) regions. 

The article is organized as follows. In the next section, we introduce the governing equations, the geometry of the domain, and advocate for coupled simplified model. 
Our \textit{a posteriori} error is presented in \cref{sec:estimator}.   In \cref{sec:bounds}, we prove the global upper and lower bounds for the error with respect to the estimator, called respectively the global reliability and efficiency of the estimator. The reliability is guaranteed, i.e. the upper bound does not contain any unknown constants. Its proof is quite straightforward. The main technical hurdle is in the proof  of the global efficiency of the estimator. 
Finally, we propose in \cref{sec:numerical} an algorithm that uses our estimator both to choose the interface position and to make the local mesh refinement. The section is concluded by numerical tests. In \cref{sec:conclusions} we provide conclusions and forthcoming works. 

\section{Original problem and approximated coupled problem}\label{sec:approximation}
Let $\Omega\subset\R^2$ be the polygonal bounded domain, as shown in \cref{fig:coupledDomain}, resembling  a portion of the  channel in a real fuel cell from \cref{fig:RealThing}. Domain $\Omega$ is splitted into two parts $\Omegap$ and $\Omegat$, separated by the interface $\Gamma$, which is placed at coordinate $x_\Gamma\in (0,L)$, with $L>0$ representing the length of the straight part of the channel. Thus, $\Omegap=(0,x_{\Gamma}) \times (0,R)$, with $R>0$ standing for the width of the channel. 
The boundary $\partial \Omega$ of $\Omega$ is partitioned in the inlet $\GammaIn$ and outlet $\GammaOut$ parts, where we impose a Poiseuille-like profile, and in the remaining part $\GammaWall$, where we impose the homogeneous Dirichlet conditions.
We consider the problem with the scalar unknown $u$ on $\Omega$ such that 
\begin{subequations}\label{laplace}
\begin{align}
  \label{eq:laplace}  -\LAPL u &= f \text{ in } \Omega \,,\\
  \label{eq:laplaceBC}  u &= g \text{ on } \partial\Omega\,,
\end{align}
\end{subequations}
where 
$f = \frac{12 \uAv}{R^2}$, and 
\begin{equation}\label{eq:bc}
g = \left\lbrace
	\begin{array}{ll}
		 S(y)  & \text{ on } \GammaIn \,,\\
		 0              & \text{ on } \GammaWall \,,\\
		 S(y+W+R) & \text{ on } \GammaOut\,,\\
	\end{array}
\right.
\end{equation}
where $W>0$ stands for the straight part of the bend, cf. \cref{fig:coupledDomain}, and 
$$
S(y) =\frac{6 \uAv}{ R^2 } y (R-y)\,,
$$
with $\uAv$ the average value of the solution on the inlet/outlet. 
The weak formulation is then given by: find $u \in H_{g}^{1}(\Omega):=\{v\in H^1(\Omega):u=g\text{ on }\partial\Omega\}$ such that 
\begin{equation}\label{eq:laplaceWeak}
    \int_\Omega \nabla u \cdot \nabla v = \int_\Omega f v \quad\text{ for all } v \in H_{0}^{1}(\Omega)\,.
\end{equation}

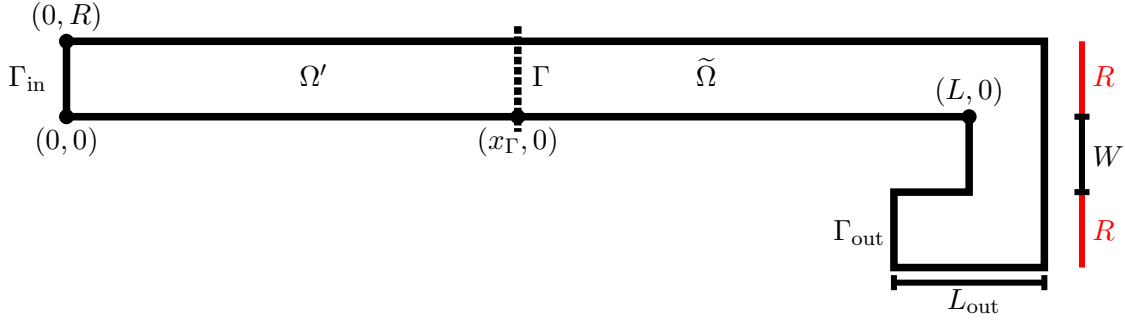
\begin{figure}[!htbp]
    \centering
    \begin{tikzpicture}
    \draw  [line width=1mm,   black]
    (-6,1) -- (6,1) -- (6,0) -- (5,0) -- (5,-1)-- (7,-1)-- (7,2)-- (-6,2)-- cycle;
    \draw  [line width=1mm,   black,densely dotted]
    (0,2.2) -- (0,0.8);
    \draw  [line width=0.8mm,   black]
    (5,-1.2) -- (7,-1.2);
    \draw  [line width=0.8mm,   black]
    (5,-1.3) -- (5,-1.1);
     \draw  [line width=0.8mm,   black]
    (7,-1.3) -- (7,-1.1);
    \draw  [line width=0.8mm,   black]
    (7.5,1) -- (7.5,0);
    \draw  [line width=0.8mm,   red]
    (7.5,2) -- (7.5,1);
    \draw  [line width=0.8mm,   red]
    (7.5,0) -- (7.5,-1);
    \draw  [line width=0.8mm,   black]
    (7.4,1) -- (7.6,1);
     \draw  [line width=0.8mm,   black]
    (7.4,0) -- (7.6,0);
    \tkzDefPoint(-6,1){A}
    \tkzLabelPoint[left,below](A){$(0,0)$}
    \node at (A)[circle,fill,inner sep=2pt]{};
    \tkzDefPoint(-6,2){B}
    \tkzLabelPoint[left,above](B){$(0,R)$}
    \node at (B)[circle,fill,inner sep=2pt]{};
    \tkzDefPoint(0,1){C}
    \tkzLabelPoint[below](C){$(x_\Gamma,0)$}
    \node at (C)[circle,fill,inner sep=2pt]{};
    \tkzDefPoint(6,1){D}
    \tkzLabelPoint[above](D){$(L,0)$}
    \node at (D)[circle,fill,inner sep=2pt]{};
\end{tikzpicture}
    \put(-300,78){$\Omegap$}
    \put(-150,78){$\Omegat$}
    \put(-212,78){$\Gamma$}
    \put(-410,78){$\GammaIn$}
    \put(-98,20){$\GammaOut$}
    \put(0,48){$W$}
    \put(-55,-6){$L_{\rm out}$}
    \put(0,78){\textcolor{red}{$R$}}
    \put(0,20){\textcolor{red}{$R$}}
    \caption{Partition of the domain $\Omega$ in domains for the coupled 0D/2D model. $x_\Gamma$ is the $x$-coordinate of the interface position, $L, R, W, L_{\rm out} >0$. 
    }
    \label{fig:coupledDomain}
\end{figure}
\begin{figure}[!htbp]
	\centering
	\includegraphics[scale=0.2]{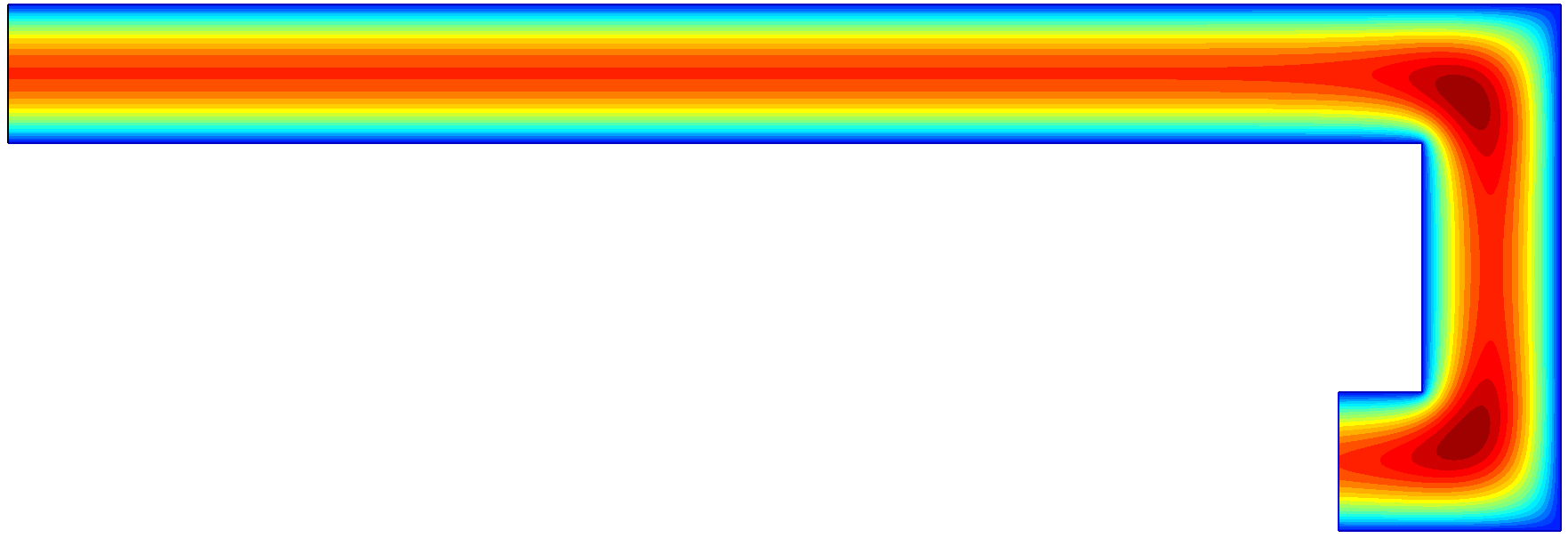}
	\includegraphics[scale=0.5]{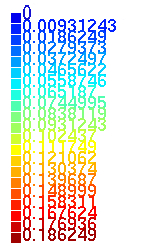}
	\caption{Solution to problem \eqref{laplace} approximated by $\Po{2}$-FE on a mesh of size approximately $h = 0.005$, used as reference solution for numerical tests in \cref{sec:numValidation}. 
	}
	\label{fig:solution}
\end{figure}

A typical solution to this problem is shown at \cref{fig:solution}. We observe that it is indeed of simple form, independent of $x$, in the rectangular portion of the channel, to the left of the bend. It is where we place the subdomain $\Omegap$. Let us now derive a simplified model in $\Omegap$ using techniques as in \cite{panasenko1998method, tayachi2014design}. 
Let us choose as characteristic constants $R$, $L$, $\uAv$ for $x$, $y$, $u$, and consider the dimensionless quantities: 
$\bar{u}=\frac{u}{\uAv}, \bar{x}=\frac{x}{L}, \bar{y}=\frac{y}{R}$. Assuming $R \ll L $, let us also define $\epsilon=\frac{R}{L}$.
Then, rewriting \eqref{eq:laplace} on $\Omegap$ in dimensionless variables, we get 
\begin{equation}\label{eq:laplaceAdim}
 -\epsilon^2\partial^2_{\bar{x}\bar{x}}\bar{u} - \partial^2_{\bar{y}\bar{y}}\bar{u}  =\bar{f} \quad \text{for} \quad  \bar{y} \in (0,1) \,,
\end{equation}
with $\bar{f}=\frac{R^2f}{\uAv} =12$. 
Neglecting the term multiplied by $\epsilon^2 \ll 1$ and recovering the variables with dimensions, we get the approximate equation on $\Omegap$ whose solution is denoted  by $u'$: 
\begin{equation}\label{eq:laplaceModel0D}
-\partial^2_{yy}u'  = 12\frac{u_{av}}{R^2}  \quad \text{in} \quad \Omegap \,.
\end{equation}
Combining this with the boundary conditions on $\GammaWall$, i.e. $u'=0$ for $y=0$ and $y=R$, leads to the approximation  
\begin{equation}\label{eq:up}
    \up=S(y) \text{ in } \Omegap\,. 
\end{equation}
This approximation is evidently compatible with the inlet boundary condition on $\GammaIn$.

In $\Omegat$, we shall solve numerically \eqref{laplace} by the FE method.  
Let $\Th$ be a triangular mesh on $\Omegat$ and $V_h$ the FE space of piecewise $\Po{k}$ polynomials with $k\geq 2$, i.e. $V_h=\{v_h\in H^1(\Omegat):{v_h}_{|_K}\in\Po{k}\,\ \forall K\in\Th\}$. The simplest coupling condition between the approximation $u'$ on $\Omegap$ and  the numerical approximation $\uht$ on $\Omegat$ is $\uht = \up$ on $\Gamma$. We thus search $\uht \in V_h$ such that 
$\uht=g$ on $\partial\Omegat$, and
\begin{equation}\label{eq:laplaceOmegatWeak}
   \int_{\Omegat} \nabla \uht \cdot \nabla v_h = \int_{\Omegat} f v_h, \quad \forall v_h \in V_h, v_h = 0 \text{ on } \partial \Omegat\,.
\end{equation}

Finally, the coupled approximated solution $\uhc$ is defined as 
\begin{equation}\label{eq:uhc}
\uhc = \left\{ \begin{array}{l}
     \up \text{ in } \Omegap\,,\\
     \uht \text{ in } \Omegat\,.
   \end{array} \right. 
\end{equation}

\section{\textit{A posteriori} error estimator}
\label{sec:estimator}

In this section, we introduce an \textit{a posteriori} error estimator with guaranteed reliability and provable efficiency for the error defined by 
\begin{equation}\label{eq:err}
    \err = \norm[\Omega]{  \nabla u - \nabla \uhc}\,
\end{equation} 
in the spirit of \cite{ern2015polynomial}. 
To this end, we should define a reconstructed flux $\sigmahc \in  H(\mathrm{div} , \Omega)$ such that $\nabla\cdot \sigmahc = f$ in $\Omega$ and such that its construction be local. This is done below, cf. \eqref{eq:sigmaCoupled}, as the sum of local contributions from the patches around the mesh nodes and a patch attached to the interface $\Gamma$.  Then, the error estimator will be defined as 
\begin{equation}\label{eq:eta}
\eta = \norm{ \sigmahc + \nabla \uhc } \,.
\end{equation}

Let $\Nh$ and $\NhGamma$ be the sets respectively of all the nodes of the mesh $\Th$ of $\Omegat$ and of all the nodes lying on the interface $\Gamma$. 
Let $\psi^{\Gamma}$ be the continuous function on $\Omega$ defined on rectangular portion $[0,x_{\Gamma}+R]\times [0,R]$ of $\Omega$, supposing that $x_{\Gamma}+R\leq L$, by 
$$
\psi^{\Gamma} (x,y) = 
    \left\{ \begin{array}{ll}
        1, \quad &\text{for } x < x_{\Gamma} \,, 0 \leq y \leq R \,,\\
        \frac{x_{\Gamma} + R - x}{R}, \quad &\text{for }  x \in [x_{\Gamma}, x_{\Gamma} + R] \,, 0 \leq y \leq R \,,\\
        0, \quad &\text{for }  x > x_{\Gamma} + R\,, 0 \leq y \leq R \,,
   \end{array} \right.
$$
and extended by 0 everywhere else. 
We also introduce a piecewise affine on mesh $\Th$ version of $\psi^{\Gamma}$, defining 
\begin{displaymath}
\psi^{\Gamma}_h (x,y) = 
    \left\{ \begin{array}{ll}
    \sum_{\n \in \Nh } \psi^{\Gamma} (\n) \psi^\n  (x,y) \ &\text{ if } (x,y) \in \Omegat \cup \Gamma \,,\\
    1 &\text{ if } (x,y) \in \Omegap\,,
   \end{array} \right.
\end{displaymath}
where $\psi^\n$ is the hat function associated to node $\n\in\Nh$, i.e. the piecewise $\Po{1}$ function taking the value 1 at $\n$ and vanishing at all the other mesh nodes. 
It follows that we have the partition of unity on $\Omega$:
\begin{equation}\label{eq:partitionUnity}
1 = \psi^{\Gamma}_h + \sum_{\n\in\NhWGamma} (1 - \psi^{\Gamma} (\n)) \psi^\n\,.
\end{equation}
\begin{remark}\label{rmk:psiGamma}
If $x_{\Gamma} + R>L$, then the function $\psi^{\Gamma}$ defined above becomes discontinuous and this prevents us from proving the optimality estimate \eqref{eq:lower} in the manner of Theorem \ref{thm:globalBounds}. However, $\psi_h^\Gamma$ remains continuous even if $x_{\Gamma} + R>L$. Partition of unity also remains valid, so that the construction of the flux can be carried out as explained below. Hence, this hypothesis $x_{\Gamma} + R \le L$ is not necessary for the implementation of our error estimator. The numerical results show that it remains robust even if this theoretical assumption is violated. 
\end{remark} 

Let us define the patches $\omega^\n = \operatorname{supp} \psi^\n$ for all the nodes $\n\in\Nh$ and the patch attached to $\Gamma$: $\omega^{\Gamma} = \operatorname{supp} (\psi_{h}^{\Gamma}) \cap \Omegat$. Consider the three following cases to define the $\mathbf{H}(\DIV)$-conforming Raviart-Thomas piecewise polynomials spaces on patches. 
We remind that $\bn$ denotes the outward unit normal vector of patches. 
For the case where $\n$ is an internal node of $\Omegat$, that is $\n\in\Nh\cap\Omegat$, we define 
$$
\begin{array}{l}
\SigmaN = 
\{ \tauh \in \mathbf{H}(\DIV,\omega^\n) ; {\tauh}_{|\K} \in \mathbf{RT}_k(\K) \  \forall \K \in \omega^\n,  \tauh\cdot\mathbf{n}=0 \text{ on } \partial\omega^\n\}\,, \\
\QN  = \{ q_h \in L^2(\omega^\n) ; {q_h}_{|\K} \in \Po{k}(\K)  \  \forall \K \in \omega^\n , \int_{\omega^\n} q_h =0 \}\,,
\end{array}
$$
where $\mathbf{RT}_k(\K)=[\Po{k}(\K)]^2+\mathbf{x}\Po{k}(\K)$ with $\K$ a mesh element.  
For the case where $\n$ is a node on the boundary of $\Omegat$ excluded the interface $\Gamma$, that is  $\n\in\NhWGamma\cap\partial\Omegat$, we define 
$$
\begin{array}{l}
\SigmaN = \{ \tauh \in \mathbf{H}(\DIV,\omega^\n) ; {\tauh}_{|\K} \in \mathbf{RT}_k(\K) \  \forall \K \in \omega^\n ,  \tauh\cdot\mathbf{n}=0 \text{ on } \partial \omega^\n \cap \Omegat \}\,, \\
\QN  = \{ q_h \in L^2(\omega^\n) ; {q_h}_{|\K} \in \Po{k}(\K)  \  \forall \K\in\omega^{\n} \}\,.
\end{array}
$$
Finally, on the patch $\omega^\Gamma$ attached to the interface $\Gamma$, we define
$$
\begin{array}{l}
\SigmaGammah = 
\{ \tauh \in \mathbf{H}(\DIV,\omega^\Gamma) ; {\tauh}_{|\K} \in \mathbf{RT}_k(\K) \  \forall \K \in \omega^\Gamma , \tauh\cdot\mathbf{n}=0 \text{ on } \Gamma\cup\widetilde{\Gamma}\,\},\\
\QGamma  = \{ q_h \in L^2(\omega^\Gamma) ; {q_h}_{|\K} \in \Po{k}(\K)  \  \forall \K \in \omega^\Gamma   \}\,,
\end{array}
$$
with $\widetilde{\Gamma}=\partial \omega^\Gamma \cap \Omegat$. 
These are the appropriate spaces to define reconstructed fluxes, cf. the following  \cref{def:sigmaN}. This definition is completely standard, cf. \cite{ern2015polynomial}, in the case of patches centered at the mesh nodes, but we extend it also to the patch attached to the interface $\Gamma$. This is a novelty of the present work. 
\begin{definition}\label{def:sigmaN}
For any mesh node $\n\in\NhWGamma$ or for $\n=\Gamma$, recalling the corresponding spaces $\SigmaN$ and $\QN$ defined above, 
let $\sigmaN \in \SigmaN$ and $p_{h}^{\n} \in \QN$ be the solution to  
\begin{subequations}
\begin{align}
  \label{eq:mixa1} & \int_{\omega^\n} \sigmaN \cdot \tauh + 
                            \int_{\omega^\n}  p_{h}^{\n} \DIV \tauh = 
                            -\int_{\omega^\n} (\nabla \uht ) \psi^\n \cdot \tauh   
                            &&\forall \tauh \in \SigmaN \,,\\
  \label{eq:mixa2} & \int_{\omega^\n} q_h \DIV \sigmaN = 
                        \int_{\omega^\n} (f\psi^\n - \nabla \uht \cdot \nabla \psi^\n) q_h  
                        &&\forall q_h \in \QN\,.
\end{align}
\end{subequations}
Each $\sigmaN$ is extended by 0 on $\Omegat$ outside $\omega^\n$.
\end{definition}
\begin{remark}\label{rmk:fctTest}
Relation \eqref{eq:mixa2} is actually satisfied for all the test functions $q_h$ in the discontinuous $\Po{k}$ space, without the constraint of vanishing average, even at the internal nodes $\n\in\Nh\cap\Omegat$. Indeed, taking any internal node $\n$, we have thanks to the divergence theorem and $\sigmaN\in\SigmaN$,  
\begin{equation}\label{eq:fctTest1}
\int_{\omega^\n} \DIV \sigmaN = - \int_{\partial \omega^\n} \sigmaN\cdot \bn = 0\,.
\end{equation}
Thanks to \eqref{eq:laplaceOmegatWeak} with $v_h=\psi^\n$, we also have  
\begin{equation}\label{eq:fctTest2}
\int_{\omega^\n} ( f \psi^\n - \nabla \uht \cdot \nabla \psi^\n )\; 1 = 0\,.
\end{equation}
Hence, \eqref{eq:mixa2} is satisfied for $q_h=1$ and thus for any $q_h\in\QN+\operatorname{span}(1)$. 
\end{remark}

We now define the flux $\sigmahc$ on $\Omega$ as
\begin{equation}\label{eq:sigmaCoupled}
	\sigmahc =     \left\{ \begin{array}{ll}
		\sigmaP &\text{ in } \Omegap  \,,\\
		\sigmat &\text{ in } \Omegat\,,
	\end{array} \right.
\end{equation}
where 
\begin{equation}\label{eq:sigmap} 
	\sigmaP=-\nabla \up \text{ in } \Omegap\,,
\end{equation}
and
\begin{equation}\label{eq:sigmat}
	\sigmat = \sigmaGammah + \sum_{\n\in\NhWGamma} (1 - \psi^{\Gamma} (\n)) \sigmaN \text{ in } \Omegat\,,
\end{equation}
with $\sigmaGammah$ and $\sigmaN$ from  \cref{def:sigmaN}.

\begin{proposition}\label{prop:divSigma}
The flux $\sigmahc$ defined in \eqref{eq:sigmaCoupled} is such that 
$\sigmahc\in \mathbf{H}(\DIV,\Omega)$ and 
$\DIV \sigmahc =  f$ in $\Omega$, 
where $f$ is the source term of \eqref{laplace}.  
\end{proposition}
\begin{proof}
On $\Omegap$, we obviously have	
\begin{equation}\label{eq:divSigmaP}
    \DIV \sigmaP =  -\Delta u' = f \text{ in } \Omegap\,.
\end{equation}
In particular, $\sigmaP\in \mathbf{H}(\DIV,\Omegap)$.

On $\Omegat$, we have $\sigmaN\in \mathbf{H}(\DIV,\Omegat)$ for any $\n\in\NhWGamma$ or $\n=\Gamma$ thanks to the boundary conditions in the definition of $\SigmaN$ and to the extension by 0 outside $\omega^\n$. Moreover, thanks to \cref{rmk:fctTest}, equation \eqref{eq:mixa2} can be rewritten as
$$
 \int_{\omega^\n} q_h \DIV \sigmaN = 
\int_{\omega^\n} (f\psi^\n - \nabla \uht \cdot \nabla \psi^\n) q_h  
\quad\forall q_h \in \mathcal{Q}_h\,,
$$
where $\mathcal{Q}_h=\Po{k}(\Th)$ is the broken polynomial space of degree $k\geq 2$ on the mesh $\Th$. Thus, by \eqref{eq:sigmat}, 
\begin{align*}
\int_{\Omegat} ( \DIV \sigmat ) q_h& =     \int_{\Omegat} ( \DIV \sigmaGammah ) q_h + 
    \sum_{\n\in\NhWGamma} (1 - \psi^{\Gamma} (\n))  \int_{\Omegat} ( \DIV \sigmaN ) q_h  \\
    &=\int_{\Omegat} \Big( f ( \psi^\Gamma_h + \sum_{\n\in\NhWGamma} (1 - \psi^{\Gamma} (\n))\psi^\n  ) 
                    - \nabla \uht \cdot \nabla ( \psi^\Gamma_h + \sum_{\n\in\NhWGamma} (1 - \psi^{\Gamma} (\n))\psi^\n  )  \Big) q_h \\
    &= \int_{\Omegat}  f q_h \,,
\end{align*}
where to pass from the second to the third line, we have used the partition of unity \eqref{eq:partitionUnity}. Since $\DIV  \mathbf{RT}_k(\K) = \Po{k}(\K)$ on any mesh element $\K$, this implies 
\begin{equation}\label{eq:divSigmahc}
	\DIV \sigmat =   f \text{ in } \Omegat\,.
\end{equation}

In view of \eqref{eq:divSigmaP} and \eqref{eq:divSigmahc} and the definition (\ref{eq:sigmaCoupled}) of $\sigmahc$, it remains to prove that the normal component of $\sigmahc$ is continuous across $\Gamma$. This is evident by construction since $\sigmaP\cdot\mathbf{n}=\sigmat\cdot\mathbf{n}=0$ on $\Gamma$. 
\end{proof}

\section{Upper and lower bounds}
\label{sec:bounds}

\begin{theorem}[Global reliability and efficiency]\label{thm:globalBounds}
Let $u$ be the weak solution of \eqref{eq:laplaceWeak},  $\uhc$ be the solution \eqref{eq:uhc} of the coupled 0D/2D problem, $\err$ and $\eta$ be the error and the error estimator defined respectively by \eqref{eq:err} and \eqref{eq:eta}. Then 
\begin{equation} \label{eq:upper}
    \err \leq \eta
\end{equation}
and, assuming  $x_{\Gamma}+R\leq L$,
\begin{equation}\label{eq:lower} 
    \eta \leq C \err
\end{equation}
with a constant $C$ depending only on the mesh regularity. \end{theorem}
\begin{proof}
The proof of the upper bound \eqref{eq:upper} is completely standard: 
we set  $e = u - \uhc$, observe that $e \in H^1_0 (\Omega)$, 
use the weak formulation \eqref{eq:laplaceWeak}, \cref{prop:divSigma} and integration by parts (with the fact that $e=0$ on $\partial\Omega$), to get 
\begin{subequations}
\begin{align*}
\norm[\Omega]{\nabla u - \nabla \uhc} & = 
\int_\Omega (\nabla u - \nabla \uhc)\cdot(\nabla u - \nabla \uhc) = 
\int_\Omega (\nabla u - \nabla \uhc)\cdot \nabla e = 
\int_\Omega f e - \int_\Omega  \nabla \uhc \cdot \nabla e \\
& = \int_\Omega \DIV \sigmahc e - \int_\Omega  \nabla \uhc \cdot \nabla e =
\int_\Omega (-\sigmahc -  \nabla \uhc) \cdot \nabla e 
\leq \norm[\Omega]{\sigmahc + \nabla \uhc} \norm[\Omega]{\nabla e}\,,
\end{align*}
\end{subequations}
where for the inequality we apply Cauchy-Schwartz inequality. Hence \eqref{eq:upper}. 

The proof of the lower bound \eqref{eq:lower} needs the following four steps. In this proof, we shall use the letter $C$ for various constants that depend only on the mesh regularity (in particular, independent of the geometrical parameter $R$ and of the polynomial degree $k$).

\paragraph{Step 1: error caused by the interface, prior to discretization} 
Let us begin with a ``continuous'' version of the coupled 0D/2D model: 
\begin{subequations}
\begin{align}
  \label{eq:laplaceOmegat} & -\LAPL \tu = f \text{ in } \Omegat \,,\\
  \label{eq:laplaceOmegatBC} & \tu = 
    \left\lbrace
	\begin{aligned}
		& \up   && \text{ on } \Gamma \,,\\
		& 0     && \text{ on } \GammaWall \,,\\
		& \uOut && \text{ on } \GammaOut\,,\\
	\end{aligned}
    \right.
\end{align}
\end{subequations}
and set 
$$
\uc = \left\{ \begin{array}{l}
     \up \text{ on } \Omegap \,,\\
     \tu, \text{ on } \Omegat\,.
   \end{array} \right. 
$$
We want to study $\norm[\Omega]{\nabla u - \nabla \uc }$ which represents
the error introduced by the interface alone, without discretizing the problem
in $\Omegat$. 
We also introduce the continuous version of $\sigmaGammah$:
$\sigmaGammaCont \in H_{\Gamma} (\DIV, \omega^{\Gamma})$, $p^{\Gamma} \in L^2
(\omega^{\Gamma})$ with $H_{\Gamma} (\DIV, \omega^{\Gamma}) = \{\tau \in H
(\DIV, \omega^{\Gamma}) : \tau \cdot \bn = 0 \text{ on } \Gamma\}$, such that
\begin{subequations}\label{eq:sigmaG}
	\begin{align}
		\label{eq:sigmaG1}    
		\int_{\omega^{\Gamma}} \sigmaGammaCont \cdot \tau^{\Gamma} 
		+ \int_{\omega^{\Gamma}} p^{\Gamma} \DIV \tau^{\Gamma} 
		& = -\int_{\omega^{\Gamma}} \nabla \tu \cdot \tau^{\Gamma} 
		&& \forall \tau^{\Gamma} \in H_\Gamma (\DIV, \omega^{\Gamma}) \,,\\
		\label{eq:sigmaG2} 
		\int_{\omega^{\Gamma}} q^{\Gamma} \DIV \sigmaGammaCont 
		&= \int_{\omega^{\Gamma}} f q^{\Gamma} && \forall q^{\Gamma} \in L^2(\omega^{\Gamma}) \,.
	\end{align}   
\end{subequations}

Let us prove
\begin{equation}\label{eq:LowerIntf} 
	\norm[\omega^{\Gamma}]{ \sigmaGammaCont + \nabla \tu } \leq 
	C \norm[\Omega]{ \nabla u - \nabla \uc}\,.
\end{equation}
Let $\omega_R^{\Gamma} = \operatorname{supp} (\psi^{\Gamma}) \cap \Omegat$ and
$\Gamma_R = \partial \omega_R^{\Gamma} \cap \{x = x_{\Gamma} + R\}$. In Lemma
\ref{lemma:1} we show that there exists $\theta \in H^1 (\omega_R^{\Gamma})$
such that
\begin{equation}\label{eq:thetaGamma}	
	\begin{aligned}
		\LAPL \theta & = 0 &  & \text{in } \omega^{\Gamma}_R \,,\\
		\nabla \theta \cdot \bn & = - \nabla \tilde{u} \cdot \bn &  & \text{on }
		\Gamma \,,\\
		\nabla \theta \cdot \bn & = 0 &  & \text{on } \Gamma_R \,,\\
		\theta & = 0 &  & \text{on } \GammaWallt \cap \partial \omega_R^{\Gamma}
		\,,
	\end{aligned}
\end{equation}
and $ \norm[\omega_R^{\Gamma}]{ \nabla \theta } \leq C_1 \norm[-\nicefrac12,\Gamma]{ \nabla \tu \cdot \bn }$, with $C_1$ independent of $R$. Let
\[ \tau^c = \left\{ \begin{aligned}
	& \nabla \theta &  & \text{in } \omega^{\Gamma}_R \,,\\
	& 0 &  & \text{on } \omega^{\Gamma} \setminus \omega^{\Gamma}_R
	\, .
\end{aligned} \right. \]
By construction, $\DIV\tau^c=0$ in $\omega^{\Gamma}$. This implies
\[ \DIV (\sigmaGammaCont + \nabla \tu + \tau^c) = 0 \quad
\text{in } \omega^{\Gamma} \,\]
thanks to \eqref{eq:sigmaG2}, and
\begin{equation}
	\label{eq:divFreeFctTest} (\sigmaGammaCont + \nabla \tu +
	\tau^c) \cdot \bn = 0 \quad \text{on } \Gamma
\end{equation}
since $\sigmaGammaCont \cdot \bn = 0$ and $\tau^c \cdot \bn = - \nabla \tilde{u}\cdot \bn$ by \eqref{eq:thetaGamma} on $\Gamma$. This
allows us to take $\tau^{\Gamma} = \sigmaGammaCont + \nabla \tu + \tau^c$ as a test function in \eqref{eq:sigmaG1}. Hence, by Cauchy-Schwartz,
\begin{equation}\label{eq:eff1}
    \norm[\omega^\Gamma]{\sigmaGammaCont+\nabla \tu} \leq \norm[\omega^\Gamma]{\tau^c} 
    \leq C_1 \norm[-\nicefrac12,\Gamma]{ \nabla \tu \cdot \bn }\,. 
\end{equation}

Coming back to the error, denoting by $[ \cdot ]$ the jump on $\Gamma$, we have that for all $v\in H^1_0(\Omega)$ 
$$
\int_\Omega \nabla ( u - \uc ) \cdot \nabla v = \int_\Gamma [ \nabla ( u - \uc )  \cdot \bn ] v = \int_\Gamma \nabla\tu \cdot \bn v\,,
$$
where we have split the domain $\Omega$ in $\Omegap$ and $\Omegat$, integrated by parts, using $\LAPL(u-\uc)=0$ on $\Omegap$ and $\Omegat$, $v\equiv0$ on $\partial\Omega$, and $\nabla \up \cdot \bn =0$ on $\Gamma$. 
Now, from \cref{lemma:2} we know that for each $w\in H^{\nicefrac{1}{2}}(\Gamma)$ there exists $v\in H_{0}^{1}(\Omega)$ such that 
$v_{|\Gamma} = w$ and such that $\norm[\Omega]{\nabla v} \leq C_2 \norm[\nicefrac{1}{2},\Gamma]{w}$; with $C_2$ independent of geometrical parameters. We can take $v=\theta$ in $\omega^{\Gamma}_{R}$, with $\theta$ the solution of \cref{lemma:2} and extend it to all the domain taking \textit{e.g.} $v$ the mirror of $\theta$ in $\Omegap$. 
Thus 
$$
\int_\Gamma \nabla\tu \cdot \bn v = \int_\Omega \nabla ( u - \uc ) \cdot \nabla v  \leq \norm[\Omega]{ \nabla ( u - \uc ) }\norm[\Omega]{\nabla v } \leq C_2 \norm[\nicefrac{1}{2},\Gamma]{w}  \norm[\Omega]{ \nabla ( u - \uc ) }
$$
that implies  
$$
\norm[-\nicefrac{1}{2},\Gamma]{ \nabla\tu \cdot \bn }  \leq C_2   \norm[\Omega]{ \nabla ( u - \uc ) }  \,.
$$
Coming back to \eqref{eq:eff1} and using this latter inequality, we get the desired estimate (\ref{eq:LowerIntf}) with $C= C_1 C_2$.

\paragraph{Step 2: error caused by the interface, adding the discretization in $\Omegat$}\label{proof:step2}
Here we prove that 
\begin{equation}\label{eq:LowerIntfDisc}
    \norm[\omega^\Gamma]{\sigmaGammah+(\nabla \uht) \psi^{\Gamma}_h} \leq  C  (\norm[\omega_\Gamma]{ \nabla ( \tu - \uht ) }  +\norm[\Omega]{ \nabla ( u - \uc ) }) \,.
\end{equation}
This can be viewed as a discrete analogue of (\ref{eq:LowerIntf}). 
The proof is based on Theorem 1.2 of {\cite{ern:hal-01377007}}, but the latter cannot be applied directly due to a mismatch in boundary conditions: $\sigmaGammah$ is required to vanish on both $\Gamma$ and $\widetilde{\Gamma}$, while $\psi^{\Gamma}_h$ vanishes only on $\widetilde{\Gamma}$. 
To circumvent this difficulty, we enlarge $\omega^{\Gamma}$ to $\omega^{\Gamma, m}$ by adding to $\omega^{\Gamma}$ its
mirror image with respect to $\Gamma$. 
Similarly, we extend $\psi^{\Gamma}_h, \uht, \sigmaGammah$ from $\omega^{\Gamma}$ to $\psi^{\Gamma, m}_h, \uht^m, \sigma_h^{\Gamma, m}$ \ on $\omega^{\Gamma, m}$ as functions symmetric with respect to $\Gamma$. 
Let $\widetilde{\Gamma}^m$ denote the mirror image of $\widetilde{\Gamma}$. 
Theorem 1.2 of {\cite{ern:hal-01377007}} can be formulated
on $\omega^{\Gamma, m}$ as
\begin{equation}\label{th1.2}
	\hspace{-8mm}
	\min_{\scriptsize{\begin{array}{c}
				\mathbf{v}_h \in H (\tmop{div}, \omega^{\Gamma, m}) \cap \tmop{RT}_k\\
				\tmop{div} \mathbf{v}_h = \psi^{\Gamma, m}_h f - \nabla \psi^{\Gamma,
					m}_h \cdot \nabla \uht^m\\
				\mathbf{v}_h = 0 \tmop{on} \tilde{\Gamma} \cup \widetilde{\Gamma }^m
	\end{array}}} \hspace{-3mm}\left\| \mathbf{v}_h + \psi^{\Gamma, m}_h \nabla \uht^m
	\right\|_{\omega^{\Gamma, m}} \leqslant C (\omega^{\Gamma, m}, \psi^{\Gamma,
		m}_h) \min_{\scriptsize{\begin{array}{c}
				\mathbf{v} \in H (\tmop{div}, \omega^{\Gamma, m})\\
				\tmop{div} \mathbf{v}  = f
	\end{array}}} \left\| \mathbf{v}  + \nabla \uht^m \right\|_{\omega^{\Gamma,m}}
\end{equation}
where $C (\omega^{\Gamma, m}, \psi^{\Gamma, m}_h) \le C (\| \psi^{\Gamma, m}_h
\|_{\infty} + C_P \| \nabla \psi^{\Gamma, m}_h \|_{\infty})$ with $C_P$ the
Poincaré constant of the space $H^1 (\omega^{\Gamma, m})$ under the
constraint of functions vanishing on $\widetilde{\Gamma} \cup \widetilde{\Gamma}^m$,
i.e. $\|v\|_{\omega^{\Gamma, m}} \le C_P  \| \nabla v\|_{\omega^{\Gamma, m}}$
for all $v \in H^1 (\omega^{\Gamma, m})$ with $v = 0$ on $\widetilde{\Gamma} \cup
\widetilde{\Gamma}^m$ (note that these constraints are imposed on the part of
$\partial \omega^{\Gamma, m}$ where $\psi^{\Gamma, m}_h$ vanishes). By our
geometrical assumptions, $C_P$ is of order $R$ and $\| \nabla \psi_h^{\Gamma}
\|_{\infty} \le \frac{C}{R}$ so that $C (\omega^{\Gamma, m}, \psi^{\Gamma,
	m}_h) \le C$. Comparing with the definition of $\sigma_h^{\Gamma}$ we see that
the minimum on the left-hand side of (\ref{th1.2}) is attained on
$\sigma_h^{\Gamma, m}$ (note in particular that $\sigma_h^{\Gamma, m} \cdot n
= 0$ on $\Gamma$ by symmetry). In order to identify the minimum on the other
side, we introduce $\sigmaGammaSD \in H_{\Gamma} (\DIV, \omega^{\Gamma})$,
$\hat{p}^{\Gamma} \in L^2 (\omega^{\Gamma})$ such that
\begin{subequations}\label{sigmaHat}		
		\begin{align}
			\int_{\omega^{\Gamma}} \sigmaGammaSD \cdot \tau^{\Gamma} +
			\int_{\omega^{\Gamma}} \hat{p}^{\Gamma} \DIV \tau^{\Gamma} & = -
			\int_{\omega^{\Gamma}} \nabla \uht \cdot \tau^{\Gamma} &  & \forall
			\tau^{\Gamma} \in H_{\Gamma} (\DIV, \omega^{\Gamma}) \,, \\
			\int_{\omega^{\Gamma}} q^{\Gamma} \DIV \sigmaGammaSD & =
			\int_{\omega^{\Gamma}} fq^{\Gamma} &  & \forall q^{\Gamma} \in L^2
			(\omega^{\Gamma}) \, . 
	\end{align}
\end{subequations}
We see then that the minium on the right-hand side of (\ref{th1.2}) is
attained on $\hat{\sigma}^{\Gamma, m}$, which is the mirror extension of
$\sigmaGammaSD$ to $\omega^{\Gamma, m}$. Going back \ in (\ref{th1.2}) to the
subdomain $\omega^{\Gamma}$ of $\omega^{\Gamma, m}$ and using the symmetry
gives
\[ \left\| \sigmaGammah + (\nabla \uht) \psi^{\Gamma}_h
\right\|_{\omega^{\Gamma}} \leqslant C \left\| \sigmaGammaSD + \nabla \uht
\right\|_{\omega^{\Gamma}}\,. \]
This entails by the triangle inequality
\begin{equation}
	\label{sigmaG:3terms} \left\| \sigmaGammah + (\nabla \uht) \psi^{\Gamma}_h
	\right\|_{\omega^{\Gamma}} \le C (\left\| \sigmaGammaSD - \sigma^{\Gamma}
	\right\|_{\omega^{\Gamma}} + \left\| \sigma^\Gamma + \nabla \tilde{u}
	\right\|_{\omega^{\Gamma}} + \left\| - \nabla \tilde{u} + \nabla \uht
	\right\|_{\omega^{\Gamma}}) .
\end{equation}
To bound the first term in \ (\ref{sigmaG:3terms}), we take the difference
between \eqref{eq:sigmaG} and \eqref{sigmaHat} setting $\tau^{\Gamma} =
\sigmaGammaCont - \sigmaGammaSD$. Noting that $\DIV\tau^c = 0$, this yields
\begin{equation}
	\label{eq:errSD} \left\| \sigmaGammaCont - \sigmaGammaSD
	\right\|_{\omega^{\Gamma}} \leq  \left\| \nabla \tilde{u} - \nabla \uht
	\right\|_{\omega^{\Gamma}} \, .
\end{equation}
The second term in (\ref{sigmaG:3terms}) is bounded by (\ref{eq:LowerIntf}).
Finally, (\ref{sigmaG:3terms}) gives \eqref{eq:LowerIntfDisc}.
\paragraph{Step 3: discretization error inside $\Omegat$} We have at all the nodes $\n\in \NhWGamma$ 
\begin{equation}\label{eq:LowerNodes} 
	\norm[\omega^\n]{\sigmaN + (\nabla \uht) \psi^\n}\leq C \norm[\omega^\n]{ \nabla \tu - \nabla \uht} \,.
\end{equation}
This well known estimate follows, for example, from Theorem 1.2 of {\cite{ern:hal-01377007}} applied on each patch $\omega^\n$.

\paragraph{Step 4: putting everything together}
From definition of $\sigmahc$, see  \eqref{eq:sigmaCoupled}, \eqref{eq:sigmap} \eqref{eq:sigmat}, and the partition of unity (\ref{eq:partitionUnity}) we obtain on every mesh element $\K\in\Th$
$$
\norm[\K]{ \sigmat + \nabla \uht } \leq 
\norm[\K]{\sigmaGammah+(\nabla \uht) \psi^{\Gamma}_h}
+ \sum_{\n\in\NhWGamma} (1 - \psi^{\Gamma} (\n)) \norm[\K]{\sigmaN + (\nabla \uht) \psi^\n}
$$
where $\sigmaGammah$ and $\sigmaN$ are extended by 0 outside of their domains of definitions $\omega^\Gamma$ and $\omega^\n$ respectively. The number of non-zero terms in the sum above is thus uniformly bounded by a constant that depends only on the regularity of the mesh. Taking the squares on both sides of the inequality above leads to  
$$
\norm[\K]{ \sigmat + \nabla \uht }^2 \leq 
C\left(
\norm[\K]{\sigmaGammah+(\nabla \uht) \psi^{\Gamma}_h}^2
+ \sum_{\n\in\NhWGamma} \norm[\K]{\sigmaN + (\nabla \uht) \psi^\n}^2
\right)\,.
$$
Taking the sum over $\K\in\Th$, noting $\sigmaP+\nabla \up=0$ on $\Omegap$, and then using the bounds \eqref{eq:LowerIntfDisc}, \eqref{eq:LowerNodes} leads to 
\begin{align*}
\norm[\Omega]{ \sigmahc + \nabla \uhc } ^2
&\leq C \left(
    \norm[\omega^\Gamma]{\sigmaGammah+(\nabla \uht) \psi^{\Gamma}_h}^2 + 
    \sum_{\n\in\NhWGamma}  	\norm[\omega^\n]{\sigmaN + (\nabla \uht) \psi^\n}^2
    \right ) \\
    &\leq C \left( \norm[\Omega]{ \nabla ( u - \uc ) }^2
      +\norm[\omega^\Gamma]{ \nabla \tu- \nabla \uht }^2 +  
        \sum_{\n\in\NhWGamma} \norm[\omega^\n]{ \nabla \tu- \nabla \uht }^2 
         \right) \\
    &\leq C \left( \norm[\Omega]{ \nabla ( u - \uc ) }^2
    +\norm[\Omegat]{ \nabla \tu- \nabla \uht }^2  
    \right)       
\end{align*}
since the number of possible overlaps between the different patches $\omega^\Gamma$ and $\omega^\n$ is uniformly bounded.

By integration by parts,    
$$
    \int_{\Omegat} ( \nabla u - \nabla \tu ) \cdot ( \nabla \tu  - \nabla  \uht ) = 
        -\int_{\Omegat} (\LAPL u - \LAPL \tu ) (\tu-\uht) + \int_{\partial\Omegat} (\nabla u - \nabla \tu)\cdot \bn (\tu-\uht) = 0\,.
$$
Hence
$$
    \norm[\Omega]{ \nabla ( u - \uhc ) }^2 = 
    \norm[\Omegap]{ \nabla ( u - \up ) }^2  
     + \norm[\Omegat]{ \nabla (u-\tu) +\nabla (\tu- \uht) }^2
    = \norm[\Omega]{ \nabla ( u - \uc ) }^2 + \norm[\Omegat]{ \nabla \tu- \nabla \uht }^2\,,
$$
so that
$$
 \norm[\Omega]{ \sigmahc + \nabla \uhc } ^2 \leq 
 C \norm[\Omega]{ \nabla ( u - \uhc ) }\,,
$$
i.e. \eqref{eq:lower}.
\end{proof}

\section{Numerical results}
\label{sec:numerical}

We report here the results obtained using the 0D/2D model \eqref{eq:laplaceOmegatWeak}--\eqref{eq:uhc} for the problem  \eqref{laplace}--\eqref{eq:bc} in the domain presented in \cref{fig:coupledDomain} with $L=5.1$, $R=0.5$, $W=0.9$, and $L_{\rm out}=0.8$. All the computations are performed in FreeFEM \cite{Hecht12} and we use $\Po{2}$-Lagrange FE for the 2D model in $\Omegat$. 
Since we do not dispose of an analytical solution to \eqref{laplace}--\eqref{eq:bc}, we use a reference solution obtained with $\Po{2}$-Lagrange FEM on a fine quasi-uniform mesh on $\Omega$ with mesh size $h=0.005$ for the tests in \cref{sec:numValidation}, and on a very fine adapted mesh on $\Omega$ with 1540177 Degrees of Freedom (DoF) for the tests in \cref{sec:algo}. The reference solution on the quasi-uniform mesh is shown  in \cref{fig:solution}.

\subsection{Numerical validation on quasi-uniform meshes}\label{sec:numValidation}

In \cref{fig:test12}, we report on a series of numerical experiments varying the interface position $x_\Gamma$ from $0.1$ (near the inlet) to $5.08$ (practically in the corner), and employing quasi-uniform meshes on $\Omegat$, composed of almost equilateral triangles of approximately the same size. In \cref{fig:test1a,fig:test1b,fig:test1c}, we plot the total error $e$ \eqref{eq:err} and the total estimator $\eta$ \eqref{eq:eta} vs. the interface positions $x_\Gamma$,  using the meshes of maximal sizes $h\approx 0.08, 0.04, 0.02$ respectively. 
As expected, the ``modeling'' error caused by 0D model simplification is negligible over a wide range of interface positions. It becomes predominant only when the interface is placed very near the corner $x_\Gamma\to 5.1$. 
In \cref{fig:test2a,fig:test2b,fig:test2c}, we plot the  total error $e$ and the total estimator $\eta$ with respect to different mesh sizes $h$ and for an interface position $x_\Gamma$ fixed respectively to $0.1, 4.6$ and $5.08$. 
We choose these values because they represent the three typical cases: i)~the interface placed too far from the corner so that the meshed domain $\Omegat$ is unnecessarily big; ii)~an optimal position of the interface providing a good balance between the modeling error and the discretization error,  thus minimizing the size of the meshed domain without compromising the overall accuracy;     (iii)~interface placed too close to the corner, i.e. in the range of $x_\Gamma$ values where the error/estimator grows up abruptly in the plots on \cref{fig:test1a,fig:test1b,fig:test1c}. 
We observe the convergence under the mesh refinement in the first 2 situations ($x_\Gamma=0.1$ and $4.6$), contrary to  the last situation with $x_\Gamma=5.08$, where the convergence is lost, due to the bad choice of interface. 
When the interface is well chosen, like in \cref{fig:test2a,fig:test2b}, the rate of convergence is sub-optimal with respect to what one would expect on a smooth benchmark solution. 
This is not surprising because the actual solution is singular near the reentrant corners. 
In the next \cref{sec:algo}, we shall employ a doubly adaptive strategy (for interface position and for the mesh), which will enable us to restore an optimal rate of convergence. 

All the plots of \cref{fig:test12} confirm that  the estimator $\eta$ provides indeed an upper bound for the error $e$ and both quantities are always of the same order of magnitude, even if $x_{\Gamma}+R> L$, cf. Remark~\ref{rmk:psiGamma}. 
A study of the index of efficiency that indicates the optimality of the estimator will be carried out in the next subsection in the case of adapted meshes.

\begin{figure}[h!]
\centering
  \vskip5mm
  \begin{minipage}[b]{0.45\linewidth}\centering
    \begin{tikzpicture}[scale=0.65]
      \begin{axis}[ymin=0.002, ymax=0.098, legend style={at={(0.1,0.9)},anchor=north west} ]
        \addplot table[x=xCut,y=totErr] {images/dat/test1a.dat};
        \addplot table[x=xCut,y=totEst] {images/dat/test1a.dat};
        \legend{$e$, $\eta$};
      \end{axis} 
    \end{tikzpicture}
    \subcaption{$e$, $\eta$ vs. $x_{\Gamma}$, $h=0.0789$.\label{fig:test1a}}
  \end{minipage}
  \begin{minipage}[b]{0.45\linewidth}\centering
    \begin{tikzpicture}[scale=0.65]
      \begin{loglogaxis}[legend style={at={(0.1,0.9)},anchor=north west} ]
        \addplot table[x=meshSize,y=totErr] {images/dat/test2a.dat};
        \addplot table[x=meshSize,y=totEst] {images/dat/test2a.dat};
       \logLogSlopeTriangle{0.90}{0.4}{0.1}{1}{black};
       \legend{$e$, $\eta$};
      \end{loglogaxis}
    \end{tikzpicture}
    \subcaption{$e$ and $\eta$ vs. $h$, $x_\Gamma=0.1$.\label{fig:test2a}}
  \end{minipage}
    \begin{minipage}[b]{0.45\linewidth}\centering
   \begin{tikzpicture}[scale=0.65]
      \begin{axis}[ymin=0.002, ymax=0.098, legend style={at={(0.1,0.9)},anchor=north west} ]
        \addplot table[x=xCut,y=totErr] {images/dat/test1b.dat};
        \addplot table[x=xCut,y=totEst] {images/dat/test1b.dat};
        \legend{$e$, $\eta$};
      \end{axis} 
    \end{tikzpicture}
    \subcaption{$e$, $\eta$ vs. $x_{\Gamma}$, $h=0.0410$.\label{fig:test1b}}
  \end{minipage}
   \begin{minipage}[b]{0.45\linewidth}\centering
   \begin{tikzpicture}[scale=0.65]
      \begin{loglogaxis}[legend style={at={(0.1,0.9)},anchor=north west} ]
        \addplot table[x=meshSize,y=totErr] {images/dat/test2b.dat};
        \addplot table[x=meshSize,y=totEst] {images/dat/test2b.dat};
       \logLogSlopeTriangle{0.90}{0.4}{0.1}{1}{black};
        \legend{$e$, $\eta$};
      \end{loglogaxis}
    \end{tikzpicture}
    \subcaption{$e$ and $\eta$ vs. $h$, $x_\Gamma=4.6$.\label{fig:test2b}}
  \end{minipage}
  \begin{minipage}[b]{0.45\linewidth}\centering
   \begin{tikzpicture}[scale=0.65]
      \begin{axis}[ymin=0.002, ymax=0.098, legend style={at={(0.1,0.9)},anchor=north west} ]
        \addplot table[x=xCut,y=totErr] {images/dat/test1c.dat};
        \addplot table[x=xCut,y=totEst] {images/dat/test1c.dat};
        \legend{$e$, $\eta$};
      \end{axis} 
    \end{tikzpicture}
    \subcaption{$e$, $\eta$ vs. $x_{\Gamma}$, $h=0.02159$.\label{fig:test1c}}
  \end{minipage} 
    \begin{minipage}[b]{0.45\linewidth}\centering
    \begin{tikzpicture}[scale=0.65]
      \begin{loglogaxis}[ymin=0.01,ymax=0.2,legend style={at={(0.1,0.94)},anchor=north west} ]
        \addplot table[x=meshSize,y=totErr] {images/dat/test2c.dat};
        \addplot table[x=meshSize,y=totEst] {images/dat/test2c.dat};
       \legend{$e$, $\eta$};
      \end{loglogaxis}
    \end{tikzpicture}
    \subcaption{$e$ and $\eta$ vs. $h$, $x_\Gamma=5.08$.\label{fig:test2c}}
  \end{minipage}
  \caption{Comparison between error $e$ and error estimator $\eta$: with respect to the interface position $x_\Gamma$ for a fixed uniform mesh with mesh size $h$, see \cref{fig:test1a,fig:test1b,fig:test1c}, and with respect to the mesh size $h$ once the interface position $x_\Gamma$ is fixed, see \cref{fig:test2a,fig:test2b,fig:test2c}.\label{fig:test12}}
\end{figure}
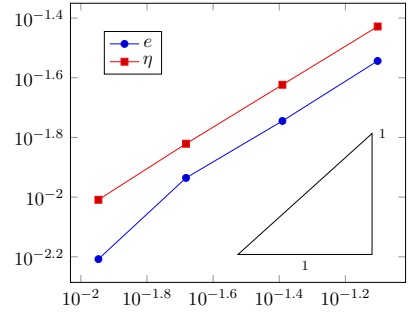
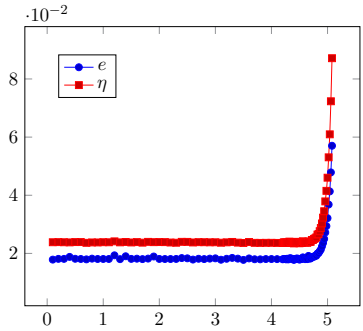
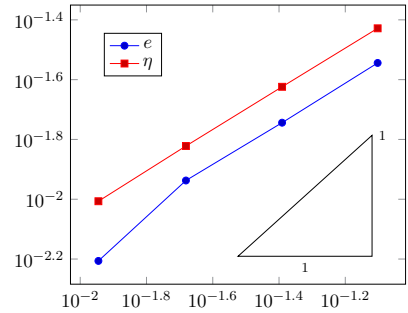
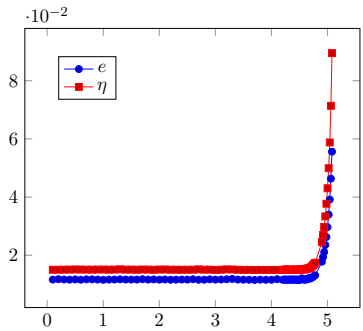
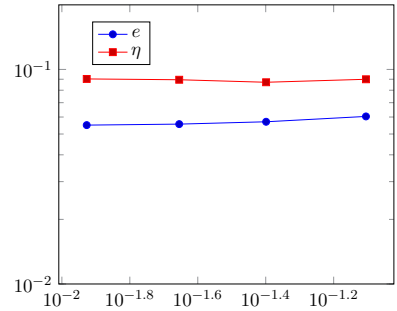

\subsection{Adaptive algorithm}\label{sec:algo}

We now inspire ourselves from the construction of the estimator \eqref{eq:eta} to detect a suitable interface position $x_\Gamma$. Guided by the proof of the upper and lower bounds for the error (Steps 2 and 3 of the proof of Theorem \ref{thm:globalBounds} in particular), we conjecture that the quantity
\begin{equation}\label{eq:etaGamma}
    \eta^\Gamma \coloneqq \norm[\omega^\Gamma]{\sigmaGammah+(\nabla \uht) \psi^{\Gamma}_h}\,,
\end{equation}
can indicate the modelling error, caused by the replacement of the 2D model by the 0D one, and can thus drive the choice of the interface position. A theoretical motivation for this conjecture stems from the estimate \eqref{eq:LowerIntfDisc} established in the proof of our main theorem. $\eta^\Gamma$ is bounded there by two terms, the second one being precisely the part of the error caused by the introduction of the coupling interface. This contribution is global. The first contribution in \eqref{eq:LowerIntfDisc} represents a discretization error, locally on the patch $\omega^\Gamma$. This local contribution should be negligible in most practical cases, in comparison with the global one.     
A numerical evidence for the pertinence of $\eta^\Gamma$ is provided  in \cref{fig:test12EtaOmageGamma}. We plot there $\eta^\Gamma$ vs. the interface position (using quasi-uniform meshes of the same size), and $\eta^\Gamma$ vs. the mesh size, once the interface position has been fixed. We observe that $\eta^\Gamma$ is indeed almost independent of the mesh refinement, and the range of interface positions where it grows abruptly coincides with the similar region for the actual error, cf. \cref{fig:test1a,fig:test1aGamma}. We conclude that $\eta^\Gamma$ can be used to detect the part of the error caused by the interface position. 

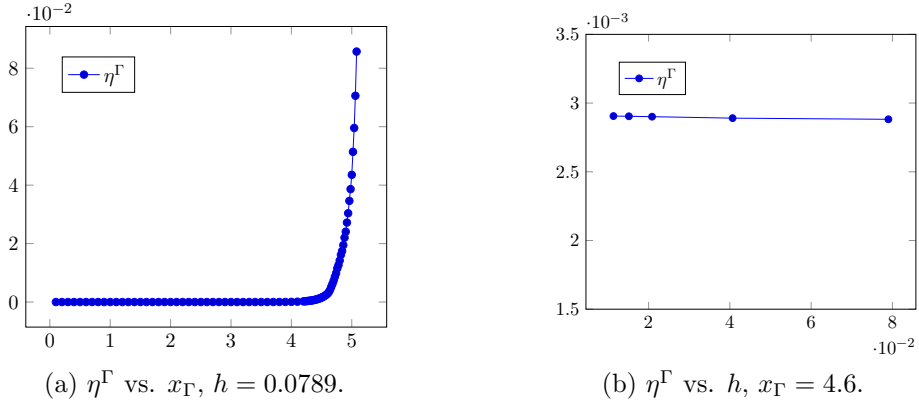
\begin{figure}
\centering
  \vskip5mm
  \begin{minipage}[b]{0.45\linewidth}\centering
    \begin{tikzpicture}[scale=0.7]
      \begin{axis}[legend style={at={(0.1,0.9)},anchor=north west} ]
        \addplot table[x=xCut,y=EstGamma] {images/dat/test1a.dat};
        \legend{$\eta^\Gamma$};
      \end{axis} 
    \end{tikzpicture}
    \subcaption{$\eta^\Gamma$ vs. $x_{\Gamma}$, $h=0.0789$.\label{fig:test1aGamma}}
  \end{minipage}
  \begin{minipage}[b]{0.45\linewidth}\centering
    \begin{tikzpicture}[scale=0.64]
      \begin{axis}[ymin=1.5E-3,ymax=3.5E-3,legend style={at={(0.1,0.9)},anchor=north west} ]
        \addplot table[x=meshSize,y=EstGamma] {images/dat/test2bBIS.dat};
       \legend{$\eta^\Gamma$};
      \end{axis} 
    \end{tikzpicture}
    \subcaption{$\eta^\Gamma$ vs. $h$, $x_\Gamma=4.6$.\label{fig:test2aGamma}}
  \end{minipage}
  \caption{Evolution of $\eta^\Gamma$ with respect to the interface position $x_\Gamma$ and a fixed mesh size $h$ and with respect to the mesh refinement with a fixed $x_\Gamma$.\label{fig:test12EtaOmageGamma}}
\end{figure}


Once the interface is fixed, it remains to refine the mesh on $\Omegat$ locally.  We use for this the error indicators defined on each mesh cell as
\begin{equation}\label{eq:etaD} 
\eta^\mathrm{D}_{\K} \coloneqq 
\norm[\K]{ (\mathbf{\sigma}_{h}^\mathrm{D}   +\nabla \uht )_{|\K} }\text{ for all $\K\in\Th$ } ,\text{ where } 
\mathbf{\sigma}_{h}^\mathrm{D} \coloneqq \sum_{\n\in\NhWGamma} \sigmaN\,
\end{equation}
(the superscript ‘‘D'' here refers to the ‘‘discretization'' part of the error). These are the usual equilibrated flux error indicators, normally used for a problem with Dirichlet boundary conditions everywhere. This seems a reasonable choice since the ‘‘modeling'' error is supposed to be controlled at this stage by a proper choice of the coupling interface.

Our approach is summarised in \cref{algo:adaptive}. We choose first a suitable interface using the error indicator $\eta^\Gamma$ \eqref{eq:etaGamma} in the loop in \cref{algo:while1}, employing coarse quasi-uniform meshes, cf.~\cref{algo:coarse}. The goal is to find an interface position such that the error indicator $\eta^\Gamma$ \eqref{eq:etaGamma} is smaller than the prescribed tolerance $\mathrm{tol}^\Gamma$, which is in turn sufficiently smaller than the desired tolerance $\mathrm{tol}$ for the overall error. At the second stage of the Algorithm, \cref{algo:while2}, we perform the mesh adaptation loop without ever moving the interface again. Note that the mesh adaptation is guided by the local error indicators $\{ \eta^\mathrm{D}_{\K}\}_{\K\in\Th}$ from \eqref{eq:etaD}, while the overall error is estimated by the global error indicator $\eta$ from our main Theorem, which controls effectively both the modeling and discretization errors. 
The details of our approach to the mesh adaptation (i.e. constructing the new mesh using the error indicators on the current mesh) are given in \cref{sec:adaptationFreefem}.

\begin{algorithm}
\caption{Adaptive interface position and mesh refinement}\label{algo:adaptive}
\begin{algorithmic}[1]
\REQUIRE{A tolerance for the global error $\mathrm{tol}$,  a tolerance for the error due to the interface position (modelling error) $\mathrm{tol}^\Gamma$, a step to move the interface $\delta x>0$.}
\STATE{Set the interface position at coordinate $x_\Gamma=L-R$, cf. \cref{fig:coupledDomain}}
\STATE{Construct a coarse mesh on $\Omegat$, compute the coupled solution \eqref{eq:uhc} and $\eta^\Gamma$, cf. \eqref{eq:etaGamma}\label{algo:coarse}}
\WHILE{ $\eta^\Gamma > \mathrm{tol}^\Gamma$ and $ x_\Gamma > 0 $ \label{algo:while1}}
\STATE{Redefine $x_\Gamma := x_\Gamma - \delta x$}
\STATE{Construct a coarse mesh on the new $\Omegat$, compute the coupled solution and $\eta^\Gamma$}
\ENDWHILE
\STATE{Compute $\eta$, as in \eqref{eq:eta}, and $\eta^\mathrm{D}_{\K}$, as in \eqref{eq:etaD}}
\WHILE{ $\eta > \mathrm{tol}$ \label{algo:while2}}
\STATE{Construct a new mesh on $\Omegat$ based on the error indicators $\eta^\mathrm{D}_{\K}$ (cf.  \cref{sec:adaptationFreefem})}
\STATE{Compute the coupled solution $u_h^c$, $\eta$, and $\eta^\mathrm{D}_{\K}$ on the new mesh}
\ENDWHILE
\RETURN Coupled solution and $x_\Gamma$
\end{algorithmic}
\end{algorithm}

We have tested \cref{algo:adaptive} for two values of the desired tolerance: 
$\mathrm{tol}=1e-2$ and $\mathrm{tol}=1e-4$. The other parameters were set to $\delta x=0.1$ and $\mathrm{tol}^\Gamma=0.1\,\mathrm{tol}$
(numerical experiments with other choices of $\mathrm{tol}^\Gamma$, ex. $\mathrm{tol}^\Gamma=0.5\,\mathrm{tol}$, give similar results but are not reported here).  As expected, the Algorithm puts the interface further from the corner when we decrease the tolerance:  $x_\Gamma=4.4$ for $\mathrm{tol}=1e-2$ and $x_\Gamma=3.8$ for $\mathrm{tol}=1e-4$.
Convergence of the error and the estimator with respect to the number of DoF on the iterations of the mesh adaptation loop (once the interface has been chosen) is presented in \cref{tab:adapt_test} and \cref{fig:adapt_tests}. 
We report in \cref{tab:adapt_test}, on subsequent adapted meshes, the numbers of DoF, the error $e$, the estimator $\eta$, and the effectivity index $I = \nicefrac{\eta}{e}$. The algorithm needed 4 iterations with $\mathrm{tol}=1e-2$ and 7 iterations for $\mathrm{tol}=1e-4$, with $I$ close to 1 in both cases. We also observe in \cref{fig:adapt_tests} that the error and the estimator converge optimally with respect to the number of DoF, i.e. at the rate that would be expected on quasi-uniform meshes if the solution were smooth. 

\begin{table}
\begin{minipage}[b]{0.43\linewidth}\centering
\begin{tabular}{|c|c|c|c|} \hline
    DoF &  $e$ &  $\eta$ & $I$ \\ \hline
    84 & 1.4968e-1  &    1.7087e-1 & 1.1 \\
    302 &  5.1694e-2 &   6.5485e-2 & 1.3 \\
    417 & 2.9590e-2 &    3.5823e-2 & 1.2 \\
    814 & 1.3262e-2 &    1.5647e-2 & 1.2 \\
    2395 & 4.7947e-3 &    5.6709e-3 & 1.2 \\\hline
\end{tabular}
\subcaption{$\mathrm{tol}=1e-2$, $x_\Gamma=4.4$.\label{tab:adapt_test1}}
\end{minipage}
\begin{minipage}[b]{0.43\linewidth}\centering
\begin{tabular}{|c|c|c|c|} \hline
   DoF &  $e$ &  $\eta$ & $I$ \\ \hline
98   & 1.5490e-1 &    1.7616e-1 & 1.1\\
347  & 5.3231e-2 &    6.6758e-2 & 1.3\\
512  &   2.6044e-2 &    3.1548e-2 &  1.2\\
1269 &    1.0420e-2 &    1.2632e-2 & 1.2\\
2721 &    3.8437e-3 &   4.5794e-3 &   1.2\\
8388 &    1.2095e-3  &    1.4192e-3 & 1.2\\
27531 &    3.3055e-4 &   3.9007e-4 &  1.2\\
106489 & 8.42367e-5 &   9.9955e-5 &  1.2\\\hline
\end{tabular}
\subcaption{$\mathrm{tol}=1e-4$, $x_\Gamma=3.8$.\label{tab:adapt_test2}}
\end{minipage}
\footnotesize
\caption{Results of \cref{algo:adaptive} with $\mathrm{tol}=1e-2$ and $\mathrm{tol}=1e-4$.}\label{tab:adapt_test}
\end{table}
\begin{figure}[h!]
\centering
  \vskip5mm
  \begin{minipage}[b]{0.45\linewidth}\centering
    \begin{tikzpicture}[scale=0.7]
      \begin{loglogaxis}[legend style={at={(0.7,0.9)},anchor=north west} ]
        \addplot table[x=DoFs,y=ErrorOmega] {images/dat/adapt_test1.dat};
        \addplot table[x=DoFs,y=EstimatorOmegaTilde] {images/dat/adapt_test1.dat};
        \logLogSlopeTriangleBIS{0.6}{0.2}{0.1}{-1}{black};
        \legend{$e$, $\eta$};
      \end{loglogaxis} 
    \end{tikzpicture}
    \subcaption{$e$, $\eta$ vs. DoF, $x_\Gamma=4.4$.\label{fig:adapt_test1}}
  \end{minipage}
  \begin{minipage}[b]{0.45\linewidth}\centering
    \begin{tikzpicture}[scale=0.7]
      \begin{loglogaxis}[legend style={at={(0.7,0.9)},anchor=north west} ]
        \addplot table[x=DoFs,y=ErrorOmega] {images/dat/adapt_test2.dat};
        \addplot table[x=DoFs,y=EstimatorOmegaTilde] {images/dat/adapt_test2.dat};
        \logLogSlopeTriangleBIS{0.6}{0.2}{0.1}{-1}{black};
       \legend{$e$, $\eta$};
      \end{loglogaxis}
    \end{tikzpicture}
    \subcaption{$e$, $\eta$ vs. DoF, $x_\Gamma=3.8$.\label{fig:adapt_test2}}
  \end{minipage}
  \caption{Evolution of $e$ and $\eta$ vs. the number of DoF for a fixed $x_\Gamma$.\label{fig:adapt_tests}}
\end{figure}
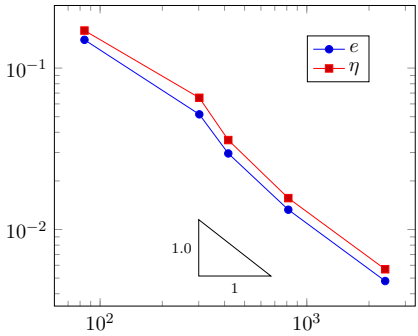
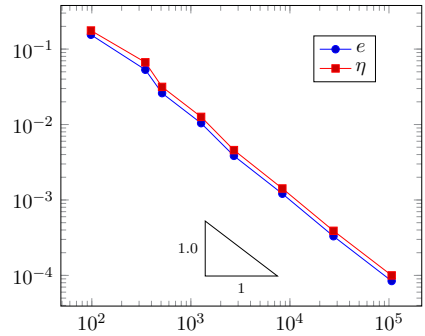

In \cref{fig:adaptedMesh}, we plot the initial coarse mesh and one of the adaptive meshes constructed by  (\cref{sec:adaptationFreefem}) with $\mathrm{tol}=1e-4$. As expected, the algorithm refines mostly near the corner where the solution is singular. 
In \cref{fig:errest5}, we present the distributions of the error and of the estimator on the same adapted mesh given in \cref{fig:adaptedMesh}. We observe that they are locally equivalent, 
confirming that our choice of the local estimators is reasonable.
At \cref{fig:adaptedMesh1}, we also observe that the error is sufficiently well equi-distributed on mesh cells (the error of about $1e-5$ per cell) with the exception of the cells near the reentrant corners, where more refinement is clearly needed. In particular, this confirms that the relatively  fine meshes in the outer corners are indeed a reasonable choice at this iteration of the mesh adaptation.

\begin{figure}[tbhp]
\subfloat[ The initial mesh. ]{\label{fig:adaptedMesh0}
\includegraphics[scale=0.28]{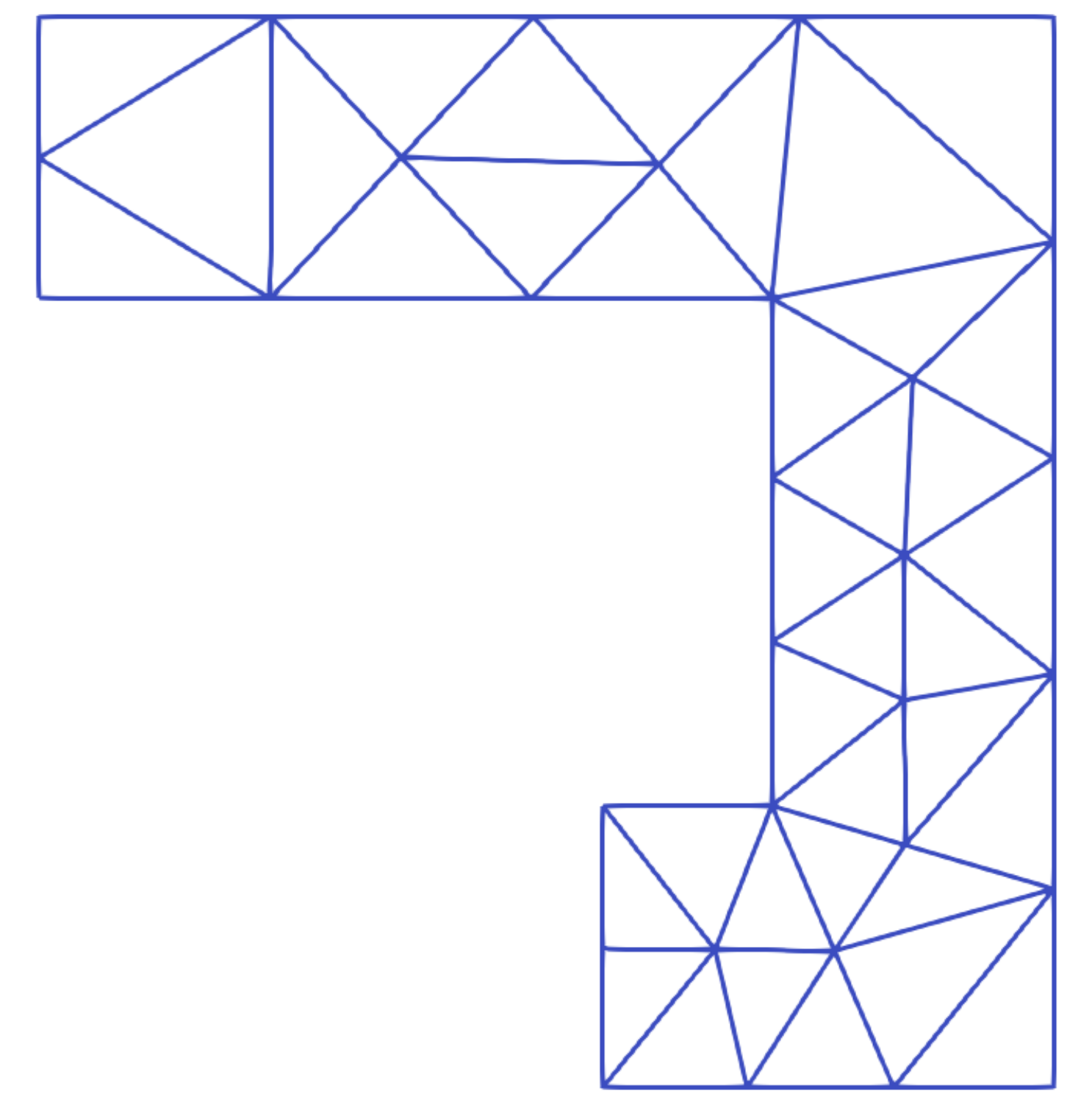}} 
 \quad
\subfloat[ The adapted mesh on iteration 5.]{\label{fig:adaptedMesh1}
\includegraphics[scale=0.3]{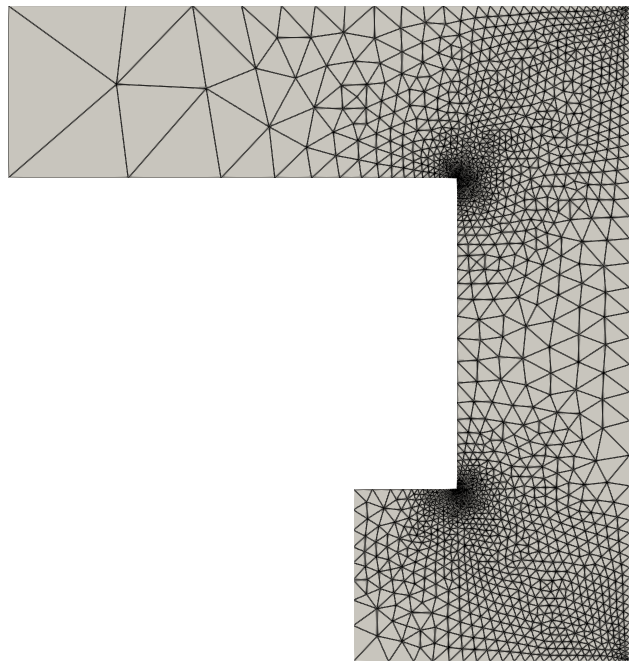}}
\caption{Two meshes on $\Omegat$ constructed by \cref{algo:adaptive} with $\mathrm{tol}=1e-4$.}\label{fig:adaptedMesh}
\end{figure}

\begin{figure}[tbhp]
\centering
\subfloat[Error distribution with a zoom in the corner.]{\label{fig:err5}
\includegraphics[scale=0.38]{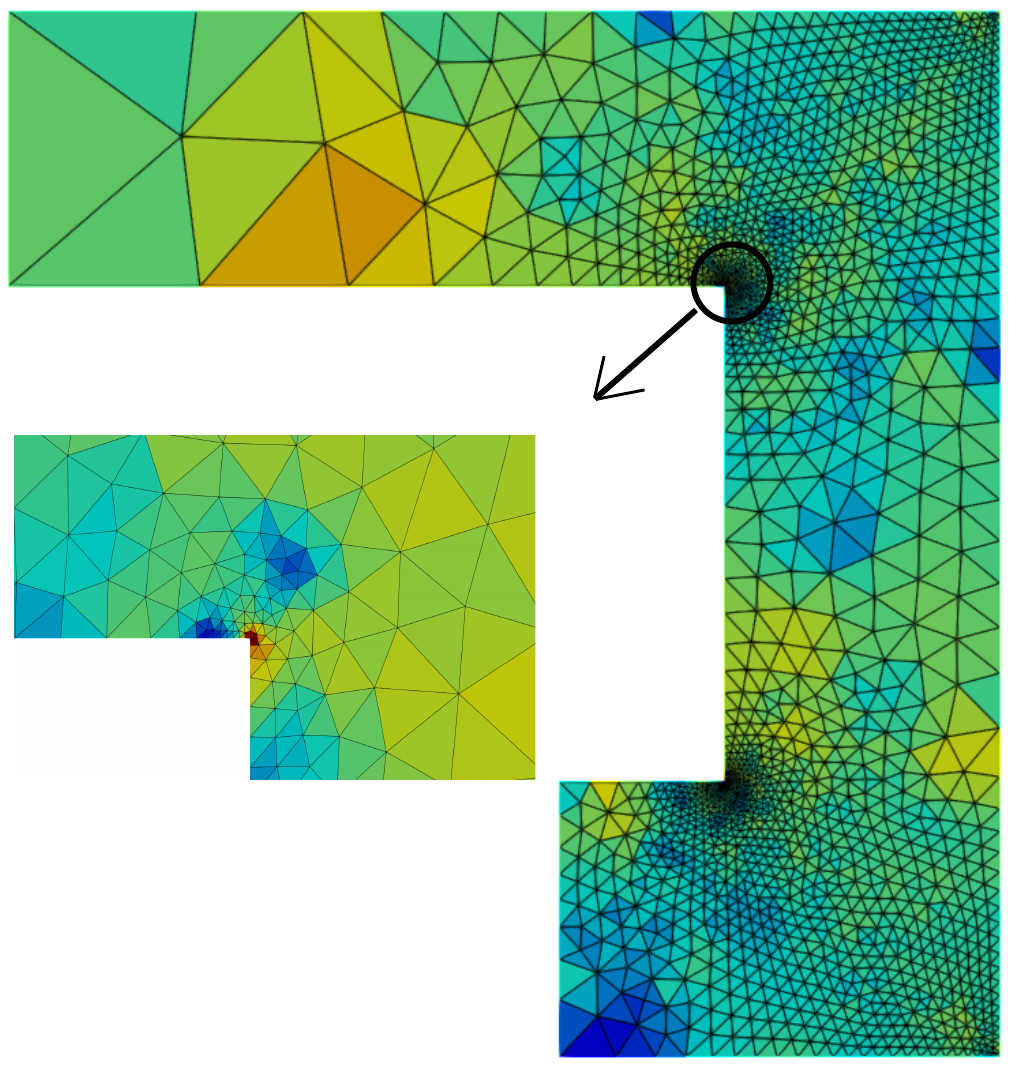}
\includegraphics[scale=0.3]{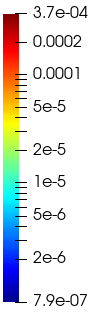}}
\subfloat[Estimator distribution with a zoom  in the corner.]{\label{fig:est5}
\includegraphics[scale=0.38]{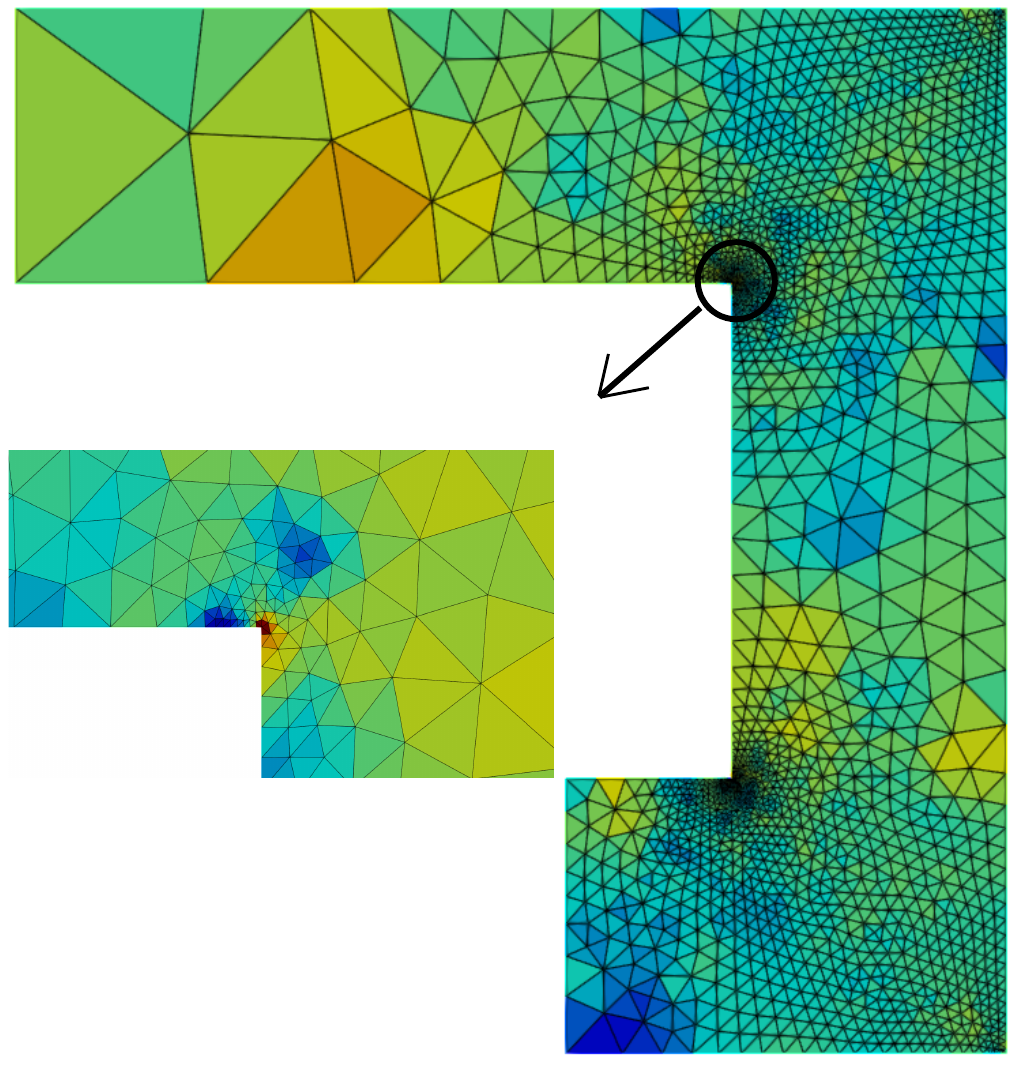} 
\includegraphics[scale=0.3]{images/err_legend}}
\caption{Distributions of the error and the estimator on iteration 5 of \cref{algo:adaptive}  with $\mathrm{tol}=1e-4$.}\label{fig:errest5}
\end{figure}

\section{Conclusions}
\label{sec:conclusions}
In this paper we have proposed a coupled approach to solve a 2D Poisson problem in a domain with long channels, splitting the domain into a portion where the solution is approximated by a simple analytical expression and another one where it is approximated numerically. 
We have developed an \textit{a posteriori} error estimator that, for a given tolerance on the error, allows us both to choose the position of the coupling interface and to adapt the mesh. Our theoretical results, namely that the estimator is guaranteed and efficient, are confirmed and enhanced by numerical experiments on automatically adapted meshes. Some partial results for the Stokes equation in a similar setting are available in \cite{HusseinPHD}. We hope to adapt this approach to more complex governing equations in a future work. 

\appendix
\section{Some technical Lemmas}
We recall here two well-known lemmas, needed for the proof of \cref{thm:globalBounds}, and give their proofs for completeness of exposition. 
Let $H_{00}^s (\Gamma)$, with $s \in \{-\frac 12,\frac 12\}$ be the spaces  of functions (distributions) on $\Gamma$ of the form 
\begin{equation}\label{eq:etaFourier} 
	\eta = \sum_{k \geqslant 1} \eta_k \sin \left( \frac{k\pi}{R} y \right)\,, 
\end{equation}
 with the norm $
	\norm[s,\Gamma]{\eta} = \left( \sum_{k \geqslant 1} \frac{R}{2}\eta_k^2 \left(\frac{k \pi}{R} \right)^{2 s} \right)^{\nicefrac{1}{2}}$.

\begin{lemma}\label{lemma:1}
	For any $\eta \in H_{00}^{-\nicefrac{1}{2}} (\Gamma)$, let $\theta$ be the solution to
	\begin{align}
		\LAPL \theta &= 0 && \text{ in } \omega^{\Gamma}_R \,,\nonumber \\
		\nabla \theta \cdot \bn &= \eta &&\text{ on } \Gamma \,, \nonumber\\
		\nabla \theta \cdot \bn &= 0 &&\text{ on } \Gamma_R \,,\nonumber\\
		\theta &= 0 &&\text{ on }  \GammaWallt\cap\partial\omega_{R}^{\Gamma}\,.\nonumber
	\end{align}
	Then 
	$\norm[\omega_R^{\Gamma}]{\nabla \theta} \leqslant C_1 \norm[- \nicefrac{1}{2}, \Gamma]{\eta}$
	with $C_1 > 0$ which does not depend on $R$.
\end{lemma}
\begin{proof}
	For any $\eta \in H_{00}^{- \nicefrac{1}{2}} (\Gamma)$ written as \eqref{eq:etaFourier}, by direct calculation, $\theta$ is given by
	$$
		\theta = \sum_{k \geqslant 1} \eta_k \sin \left( \frac{k \pi}{R} y \right) 
		\frac{\cosh \left( \frac{k \pi}{R} (x - x_{\Gamma}-R)\right)}{\frac{k \pi}{R} \sinh (k \pi)}    \,. 
	$$
	Thus, denoting $x_{\Gamma}^R=x_{\Gamma}+R$,
			\begin{multline*} 
				\norm[\omega_R^{\Gamma}]{\nabla \theta}^2 = \sum_{k \geqslant 1} \left(\frac{\eta_k }{\sinh (k \pi)}\right)^2 \times\\
				\left(
				\norm[\omega_R^{\Gamma}]{\sin \left( \frac{k \pi}{R}y \right) \sinh \left( \frac{k \pi}{R} (x-x_{\Gamma} -R) \right)}^2 
				+\norm[\omega_R^{\Gamma}]{\cos \left( \frac{k \pi}{R}y \right) \cosh \left( \frac{k \pi}{R} (x-x_{\Gamma}^R )\right) }^2 
				\right)\\
				=   \sum_{k \geqslant 1} \eta_k^2 \frac{R^2}{2k\pi\tanh(k \pi)}\leqslant \frac{1}{\tanh(\pi)}\sum_{k \geqslant 1} \eta_k^2 \frac{R^2}{2k\pi} \coloneqq  C_1^2  \| \eta \|_{ -\nicefrac{1}{2}, \Gamma}^2\,,
			\end{multline*}
	since  $\tanh(k \pi)\geqslant \tanh(\pi)$ for $k\geq 1$. 
\end{proof}
\begin{lemma}\label{lemma:2}
	For any $\eta \in H_{00}^{\nicefrac{1}{2}} (\Gamma)$, let $\theta$ be the solution to
	\begin{align}
		&\LAPL \theta = 0 \text{ in } \omega^{\Gamma}_R \,,\nonumber \\
		&\theta = \eta \text{ on } \Gamma \,, \nonumber\\
		&\theta = 0 \text{ on } \Gamma_R\cup(\GammaWallt\cap \partial\omega_{R}^{\Gamma})\,.\nonumber
	\end{align}
	Then 
	$\norm[\omega_R^{\Gamma}]{\nabla \theta} \leqslant C_2  \norm[\nicefrac{1}{2},\Gamma]{\eta}$ with $C_2 > 0$
	which does not depend on $R$.
\end{lemma}
\begin{proof}
	For any $\eta \in H_{00}^{\nicefrac{1}{2}} (\Gamma)$ written as \eqref{eq:etaFourier}, by direct calculation the solution $\theta$ is given by
	$$
		\theta = \sum_{k \geqslant 1} \eta_k \sin \left( \frac{k \pi}{R} y \right) 
		\frac{\sinh \left( \frac{k \pi}{R} (x_{\Gamma} + R - x)\right)}{\sinh (k \pi)} \,.    
	$$
	Calculations similar to those in the proof of Lemma \ref{lemma:1} lead to the desired result with $C_2=C_1$ of the previous lemma. 
\end{proof}

\section{Mesh adaptation in FreeFEM}\label{sec:adaptationFreefem}
The traditional strategy for the mesh adaptation (see, for example, \cite[Chapter 2]{verfurth} for a review) consists in marking certain mesh cells to refine according to the error indicator, and then splitting them in smaller cells. The library FreeFEM, used in our numerical experiments, does not provide tools for such a procedure: one cannot ask it to split only some specific cells, without touching the others. FreeFEM proposes instead the function \verb?adaptmesh?
that creates an entirely new mesh with  the mesh size prescribed (approximately) at every point of the entire computation domain (it is also possible to generate anisotropic meshes, but we deal here only with the isotropic version, option \verb?IsMetric=1?, where all the mesh cells are close to equilateral triangles).

We describe here an alternative mesh adaptation strategy from \cite{Olga18}, which we name ``hopt'' (for ``$h$ optimal''). 
We start from the following (admittedly not always realistic) assumption: the error between the exact solution $\tu$ and the approximate solution  $\uht$ on $\textit{any}$ mesh $\Th$ of $\Omegat$ is approximately given by
\begin{equation}\label{eq:intform}
	\norm[1,\Omegat]{ \tu - \uht}^2 \approx \int_{\Omegat} {h}^{2 \delta}(x) c^2 (x) \tmop{dx}
\end{equation}
where $h(x)$ is the mesh size distribution, i.e.  $h(x)=h_T$ on any $T \in \Th$, $\delta$ is a fixed parameter, and $c (x)$  is some \textit{a priori} unknown function.
Note also that the number of DoF is approximately given in the 2D case by
$$
N_{\tmop{DoF}} \sim \int_{\Omegat} \frac{\tmop{dx}}{h^2 (x)}\,,
$$
since a regular triangle of diameter $h$ occupies the area of order $h^2$.
Let us imagine momentarily that we know $c(x)$ and we want to construct an optimal mesh with the minimal possible $N_{\tmop{DoF}}$ to achieve a given error $\epsilon$, i.e. $\norm[1,\Omegat]{ u - \uht} = \epsilon$. This is a
constrained minimization problem for the mesh size distribution $h(x)$, i.e. minimize $\int_{\Omegat} \frac{\tmop{dx}}{h^2 (x)}$ under
$\int_{\Omegat} {h}^{2 \delta}(x) c^2 (x) \tmop{dx}=\epsilon^2$, which gives the following optimal mesh size distribution
$$
h_{\tmop{opt}} (x) =
        \frac{\epsilon^{\nicefrac{1}{\delta}}}{\left( \int_{\tilde{\Omega}}
   c^{2 / (\delta + 1)} (x) \tmop{dx} \right)^{^{\nicefrac{1}{(2 \delta)}}}}
   \frac{1}{(c(x))^{\nicefrac{1}{(\delta + 1)}}}\,.
$$
Of course, $c(x)$ is not known in practice, but, on a given mesh $\widetilde{\Th}$,
we have the estimators $\eta^\mathrm{D}_T$, see \eqref{eq:etaD}. It seems thus reasonable to expect that
$$
\norm[1,\Omegat]{ \tu - \uht}^2
\approx (\eta^\mathrm{D})^{2} := \sum_{\K \in \widetilde{\Th}} (\eta_{\K}^\mathrm{D})^{2}
    =  \sum_{\K \in \widetilde{\Th}} \norm[\K]{ \mathbf{\sigma}_{h}^\mathrm{D}  +\nabla \uht }^{2}  \,.
$$
Reinterpreting the error in the form \cref{eq:intform}, and localizing to each triangle of the current mesh $\widetilde{\Th}$, suggests
$$
 \int_{\K} h^{2 \delta} (x) c^2 (x) \tmop{dx} \sim
 (\eta_{\K}^\mathrm{D})^{2}, \quad \forall \K \in \widetilde{\Th}\,.
$$
We can thus approximate $c(x)$ on any triangle $\K \in \widetilde{\Th}$ by
$c (x) \approx \frac{\eta_{\K}^\mathrm{D}}{h_{\K}^{\delta} \sqrt{| \K |}}$, which gives
\begin{equation}\label{eq:hopt1}
h_{\tmop{opt}} (x) =
        \frac{\epsilon^{\nicefrac{1}{\delta}}}
            {\left(
            \sum_{\K \in \Th} (\eta_{\K}^\mathrm{D})^{\nicefrac{2}{(\delta + 1)}} h_{\K}^{\nicefrac{-2\delta}{(\delta + 1)}} | \K |^{  \nicefrac{\delta}{(\delta + 1)}}
            \right)^{^{\nicefrac{1}{(2 \delta)}}}}
        \frac{h_{\K}^{\nicefrac{\delta}{(\delta + 1)}} | \K |^{\nicefrac{1}{(2 \delta + 2)}}}
            {(\eta_{\K}^\mathrm{D})^{\nicefrac{1}{(\delta + 1)}}}  \quad
  \tmop{for} x \in \K\,.
\end{equation}
In practice, given a mesh $\widetilde{\Th}$, we want to reduce the error $R_e$ times (with $R_e>1$ a parameter to be specified). We thus put $\epsilon=\eta^D/R_e$ in \eqref{eq:hopt1} and give the resulting $h_{\tmop{opt}}$ mesh size distribution to the mesh-generating function of FreeFEM. This is the procedure we have used for mesh adaptation in  \cref{algo:adaptive}. There are two parameters ($\delta$ and $R_e$) to be specified. Based on some numerical experimentation (not reported here, but given in \cite{HusseinPHD}) we have chosen $\delta=1$ and $R_e=4$.

\section*{Acknowledgments}
We are grateful to Fei Gao for giving the initial impetus to this work and for sharing with us some of his expertise on fuel cells, and to  Martin Vohralik for his interest in this work and for several stimulating and enlightening discussions that has helped us to improve it. 

\bibliographystyle{siamplain}
\bibliography{references}

\end{document}